\documentclass[12pt]{article}
\usepackage{amssymb}
\usepackage{amsfonts}
\usepackage{graphicx}
\usepackage{amsmath}
\usepackage{geometry}
\usepackage{amsxtra,latexsym, amscd, enumerate}
\usepackage[mathscr]{eucal}
\usepackage{hyperref}
\usepackage{mathrsfs}
\usepackage{fancyhdr}
\usepackage{subfigure}
\usepackage{authblk}

\newtheorem{theorem}{Theorem}

\newtheorem{corollary}[theorem]{Corollary}

\newtheorem{definition}[theorem]{Definition}

\newtheorem{lemma}[theorem]{Lemma}
\newtheorem{notation}[theorem]{Notation}

\newtheorem{proposition}[theorem]{Proposition}
\newtheorem{remark}[theorem]{Remark}

\newenvironment{proof}[1][Proof]{\textbf{#1.} }{\ \rule{0.5em}{0.5em}}
\input{tcilatex}
\title{Homogenization of the spectral equation in one-dimension}

\author[1]{Thi Trang Nguyen}
\author[1]{Michel Lenczner}
\author[2]{Matthieu Brassart}
\affil[1]{FEMTO-ST Institute, 26 Chemin de l'Epitaphe, 25000 Besan\c
con, France.}
 \affil[2]{Laboratoire de Math\'ematiques de Besan\c
con, 16 Route de Gray, 25030 Besan\c con, France.}

\begin{document}
\maketitle
\begin{abstract}
The asymptotic behavior of a one-dimensional spectral problem with periodic
coefficient is addressed for high frequency modes by a method of Bloch wave
homogenization. The analysis leads to a spectral problem including both
microscopic and macroscopic eigenmodes. Numerical simulation results are
provided to corroborate the theory.

\end{abstract}

\tableofcontents

\section{Introduction}

We consider the spectral problem%
\begin{equation}
-\partial _{x}\left( {a^{\varepsilon }\partial _{x}}w^{\varepsilon }\right)
=\lambda ^{\varepsilon }\rho ^{\varepsilon }w^{\varepsilon }
\label{intro_diff}
\end{equation}%
posed in an one-dimensional open bounded domain $\Omega \subset
\mathbb{R}
$ with Dirichlet boundary conditions. An asymptotic analysis of this problem
is carried out where $\varepsilon >0$ is a parameter tending to zero and the
coefficients are $\varepsilon $-periodic, namely $a^{\varepsilon }=a\left(
\frac{x}{\varepsilon }\right) $ and $\rho ^{\varepsilon }=\rho \left( \frac{x%
}{\varepsilon }\right) $ where $a\left( y\right) $ and $\rho \left( y\right)
$ are $1$-periodic in $%
\mathbb{R}
$. The homogenization of such spectral problem has been studied in various
works providing the asymptotic behaviour of eigenvalues and eigenvectors.
The low frequency part of the spectrum has been investigated in \cite%
{kesavan1979homogenization1}, \cite{kesavan1979homogenization2}, \cite%
{vanninathan1981homogenization}. Then, many configurations have been
analyzed, as \cite{conca1988spectral} and \cite{conca1993limiting} for a
fluid-structure interaction, \cite{bal2001homogenization}, \cite%
{allaire1999homogenization} for neutron transport, \cite%
{nazarov2011homogenization}, \cite{pankratova2011homogenization} for $\rho $
which changes sign or \cite{allaire2012homogenization} for the first high
frequency eigenvalue and eigenvector for a one-dimensional non-self-adjoint
problem with Neumann boundary conditions. In \cite{allaire1998bloch}, G.
Allaire and C. Conca studied the asymptotic behaviour of both the low and
high frequency spectrum. In order to analyze the asymptotic behaviour of the
high frequency eigenvalues, they used the Bloch wave homogenization method.
It is a blend of two-scale convergence, see e.g.\textbf{\ }\cite%
{allaire1992homogenization}, \cite{allaire1993two}, \cite{lukkassen2002two},
and Bloch wave decomposition, see e.g.\ \cite{conca1997homogenization}, \cite%
{conca2005bloch}, \cite{conca2002bloch}, and was previously introduced in
\cite{allaire1996bloch} to a fluid-solid interaction problem. They have
shown that the limit of the set of renormalized eigenvalues $\varepsilon
^{2}\lambda ^{\varepsilon }$ is the union of the Bloch spectrum and the
boundary layer spectrum, when $\varepsilon $ goes to $0$. However, the
asymptotic behaviour of the corresponding eigenvectors was not addressed.
This is the goal of the present work which focuses on the Bloch spectrum of
the high frequency part. Precisely, we search eigenvalues $\lambda
^{\varepsilon }$ such that
\begin{equation}
\varepsilon ^{2}\lambda ^{\varepsilon }=\lambda _{n}^{k}+\varepsilon \lambda
^{1}+\varepsilon O\left( \varepsilon \right)  \label{intro-eigenvalue}
\end{equation}%
where $\lambda _{n}^{k}$ is solution of the Bloch wave spectral problem,
also called the microscopic equation in this work,%
\begin{equation}
-\partial _{y}\left( a\left( y\right) \partial _{y}\phi _{n}^{k}\left(
y\right) \right) =\lambda _{n}^{k}\rho \left( y\right) \phi _{n}^{k}\left(
y\right) \text{ \ for \ }n\in
\mathbb{N}
^{\ast }  \label{intro_bloch}
\end{equation}%
with $k-$quasi-periodic boundary conditions for some $k\in
\mathbb{R}
$. From \cite{allaire1998bloch}, it is known that each $\lambda _{n}^{k}$
can be reached as a limit of a subsequence of $\varepsilon ^{2}\lambda
^{\varepsilon }$. For each $n\in
\mathbb{N}
^{\ast }$ and each $k,$ $\lambda _{n}^{k}$ is either a simple or a double
eigenvalue\ and $\lambda _{n}^{k}=\lambda _{n}^{-k}$. We pose $I^{k}=\left\{
-k,k\right\} $ if $k\neq 0$ and $I^{k}=\left\{ 0\right\} $ otherwise. To
guarantee that Bloch waves are kept in the weak limit, we apply the
modulated two-scale transform $S_{k}^{\varepsilon }$, defined in \cite%
{brassart2010two} from the usual two-scale transform in \cite%
{lenczner2006homogenization}, \cite{lenczner2007multiscale}, \cite%
{cioranescu2002periodic}, \cite{casado2000two} or \cite%
{cioranescu2008periodic}.\textbf{\ }Passing to the limit in the weak
formulation, it is shown that $\sum\limits_{\sigma \in I^{k}}S_{\sigma
}^{\varepsilon }w^{\varepsilon }$ is weakly converging to two-scale modes%
\begin{equation*}
g_{k}\left( x,y\right) =\sum_{\sigma \in I^{k}}\sum_{m}u_{m}^{\sigma }\left(
x\right) \phi _{m}^{\sigma }\left( y\right)
\end{equation*}%
where the second sum runs over all modes $\phi _{m}^{\sigma }$ with the same
eigenvalue $\lambda _{n}^{k}$.\ Here, the modes $\phi _{m}^{\sigma }$ are
called microscopic modes. The factors $(u_{m}^{\sigma })_{m}$ are solution
of the macroscopic system of first order differential equation,%
\begin{equation}
\sum\limits_{m}c\left( \sigma ,n,m\right) {\partial _{x}u_{m}^{\sigma
}+\lambda }^{1}b\left( \sigma ,n,m\right) {u_{m}^{\sigma }}=0\text{ in }%
\Omega \text{ for each }\sigma \in I^{k},  \label{intro-macro}
\end{equation}%
which boundary conditions and the constant $c\left( \sigma ,n,m\right) $ are
depending on the involved microscopic modes and eigenvalues. The physical
solution $w^{\varepsilon }$ is then approximated by two-scale modes%
\begin{equation}
w^{\varepsilon }\left( x\right) \approx \sum_{\sigma \in
I^{k}}\sum_{m}u_{m}^{\sigma }\left( x\right) \phi _{m}^{\sigma }\left( \frac{%
x}{\varepsilon }\right) .  \label{intro-eigenvector}
\end{equation}%
These results are established for Neumann boundary conditions.

In fact, this method is inspired from \cite{brassart2010two} dedicated to
the wave equation, except that in the latter work the two-scale transforms $%
S_{k}^{\varepsilon }w^{\varepsilon }$ and $S_{-k}^{\varepsilon
}w^{\varepsilon }$ were analyzed separately and the macroscopic boundary
conditions were lacking. Moreover, the model derivation in \cite%
{brassart2010two} is starting from the wave equation written as a first
order system. So, for the sake of comparison, we derive the homogenized
spectral equation from a first order formulation.

In addition, we report exploration results regarding approximations of
physical eigenmodes by two-scale modes. First, for a given $\varepsilon $
and each high frequency physical eigenelement $\left( \lambda ^{\varepsilon
},w^{\varepsilon }\right) $, we show how to find quadruplets $\left( \lambda
_{n}^{k},\lambda _{1},\phi _{n}^{k},u_{n}^{k}\right) _{n,k}$ satisfying the
approximations (\ref{intro-eigenvalue}) and (\ref{intro-eigenvector}). This
shows that each high frequency eigenelement can be approximated by a
two-scale mode.\ Conversely, the high-frequency physical eigenelements can
be built from the two-scale eigenelements only. Namely, for a given Bloch
mode $\left( \lambda _{n}^{k},\phi _{n}^{k}\right) $, a macroscopic
eigenelement $\left( \lambda ^{1},u_{n}^{k}\right) $ is minimizing the error
on the physical equation (\ref{intro_diff}) where $w^{\varepsilon }$ and $%
\lambda ^{\varepsilon }$ are replaced by their approximations (\ref%
{intro-eigenvalue}) and (\ref{intro-eigenvector}).

This paper is organized as follows. In Section \ref{Statement_problem} we
state the physical spectral equation with Dirichlet boundary conditions. In
Section \ref{Notation_properties_2s} the notations and elementary
properties, which are used throughout the paper, are introduced. In Section %
\ref{second} and \ref{first}, the model homogenization is derived based on
the second order and first order formulations respectively. Finally, the
numerical results are reported in the last section.

\section{Statement of the problem}

\label{Statement_problem}

We consider $\Omega =\left( 0,\alpha \right) \subset
\mathbb{R}
^{+}$ an interval, which boundary is denoted by $\partial \Omega $, and two
functions $\left( a^{\varepsilon },\rho ^{\varepsilon }\right) $ assumed to
obey a prescribed profile,
\begin{equation}
a^{\varepsilon }:=a\left( {\frac{x}{\varepsilon }}\right) \,\ \,\ \text{and}%
\,\ \,\ \rho ^{\varepsilon }:=\rho \left( {\frac{x}{\varepsilon }}\right) ,
\label{a2_2}
\end{equation}%
where $\rho \in L^{\infty }\left(
\mathbb{R}
\right) $, $a\in W^{1,\infty }\left(
\mathbb{R}
\right) $ are both $Y$-periodic where $Y$ is an open interval. Moreover,
they are required to satisfy the standard uniform positivity and ellipticity
conditions:
\begin{equation}
\rho ^{0}\leq \rho \leq \rho ^{1}\text{ and }a^{0}\leq a\leq a^{1},
\label{a2_3}
\end{equation}%
for some given strictly positive $\rho ^{0}$, $\rho ^{1}$, $a^{0}$ and $%
a^{1} $.

With the operators $P^{\varepsilon }=-\partial _{x}\left( {a^{\varepsilon
}\partial _{x}.}\right) $, the spectral problem with Dirichlet boundary
conditions is%
\begin{equation}
{P}^{\varepsilon }{w^{\varepsilon }}=\lambda ^{\varepsilon }\rho
^{\varepsilon }w^{\varepsilon }\quad \text{in}\quad \Omega ~\text{ and }%
w^{\varepsilon }=0\quad \text{on}\quad \partial \Omega ,  \label{a2_1}
\end{equation}%
where as usual $\varepsilon >0$ denotes a small parameter intended to go to
zero.

The eigenvectors $w^{\varepsilon }\in $ $H^{2}\left( \Omega \right) \cap
H_{0}^{1}\left( \Omega \right) $ are normalized by
\begin{equation}
\left\Vert w^{\varepsilon }\right\Vert _{L^{2}\left( \Omega \right) }=\left(
\int_{\Omega }\left\vert w^{\varepsilon }\right\vert ^{2}dx\right) ^{\frac{1%
}{2}}=1,  \label{normalization}
\end{equation}
and we search the eigenvalues such that%
\begin{equation}
\varepsilon ^{2}\lambda ^{\varepsilon }=\lambda ^{0}+\varepsilon \lambda
^{1}+\varepsilon O(\varepsilon ),  \label{eigenvalue_decomposition}
\end{equation}%
where $\lambda ^{0}$ is a non negative real number and $O(\varepsilon )$
tends to zero with $\varepsilon $. The weak formulation of the spectral
problem (\ref{a2_1}) is: find $w^{\varepsilon }\in H_{0}^{1}(\Omega )$ such
that%
\begin{equation}
\int_{\Omega }a^{\varepsilon }\partial _{x}w^{\varepsilon }\partial _{x}v%
\text{ }dx=\lambda ^{\varepsilon }\int_{\Omega }\rho ^{\varepsilon
}w^{\varepsilon }v\text{ }dx\text{ \ for all }v\in H_{0}^{1}(\Omega ).
\label{1D-weakformulation1}
\end{equation}%
Since $\varepsilon ^{2}\lambda ^{\varepsilon }$ is bounded, it results the
uniform bound
\begin{equation}
||\varepsilon \partial _{x}w^{\varepsilon }||_{L^{2}(\Omega )}\leq N_{0}%
\text{.}  \label{1D-uniform-estimate}
\end{equation}

\section{Notations and elementary properties}

\label{Notation_properties_2s}

The functional space $L^{2}\left( \Omega \right) $ of square integrable
functions is over $%
\mathbb{C}
$. Let $u=\left( u_{i}\right) _{i}$ and $v=\left( v_{i}\right) _{i}$ be $m$%
-dimensional complex vector valued functions in $L^{2}\left( \Omega \right) $%
, the dot product is denoted by $u.v:=\sum\limits_{i}{u_{i}}v_{i}$ and the
hermitian inner product by%
\begin{equation}
\int_{\Omega }u\cdot v\text{ }dx=\int_{\Omega }u(x).\overline{v(x)}\text{ }%
dx.  \label{hermitian product}
\end{equation}%
The notation $O\left( \varepsilon \right) $ refers to numbers or functions
tending to zeros when $\varepsilon \rightarrow 0$ in a sense made precise in
each case. The notations $\partial _{x}u=\frac{\partial u}{\partial x}%
,\partial _{y}u=\frac{\partial u}{\partial y}$ are for $x-$ and $y-$%
derivatives of a function\ $u$. The vectors $n_{\Omega }$, $n_{Y}$ are the
outer unit normals of $\partial \Omega $ and $\partial Y.$

\textbf{Bloch decomposition} We follow the definition of Bloch decomposition
in \cite{brassart2010two} with $N=1$, $L=\mathbb{%
\mathbb{Z}
}$, and $Y=(0,1)$, so $\mathbb{R=}\overline{Y}+L.$ The dual lattice is
necessarily $L^{\ast }=\mathbb{Z}$, and the equivalence class $Y^{\ast }=%
\mathbb{R}/L^{\ast }$ is chosen as $Y^{\ast }=(-1/2,1/2)$. For $K\in \mathbb{%
N}^{\ast }$, considering the dual lattices $KL=K\mathbb{Z}$ and $L^{\ast }/K=%
\mathbb{Z}/K,$ we pose%
\begin{equation*}
L_{K}=\left\{
\begin{array}{l}
\{-\frac{K}{2},..,\frac{K}{2}-1\}\subset L\text{ if }K\text{ is even,} \\
\{-\frac{K-1}{2},..,\frac{K-1}{2}\}\text{ if }K\text{ is odd,}%
\end{array}%
\right.
\end{equation*}%
so that $L=L_{K}+KL$. Posing $L_{K}^{\ast }=L_{K}/K$ yields $L^{\ast
}/K=L^{\ast }+L_{K}^{\ast }$.

\textbf{Functional spaces of quasi-periodic functions }For any $k\in Y^{\ast
}$, we define the $k-$quasi-periodic $L^{2}-$vector space over $\mathbb{C}$
with the hermitian inner product (\ref{hermitian product}) by
\begin{equation*}
L_{k}^{2}=\{u\in L_{loc}^{2}(\mathbb{R})\text{ }{|}\text{ }u(x+\ell
)=u(x)e^{2i\pi k\ell }\text{ a.e. in }%
\mathbb{R}
\text{ for all }\ell \in L\},
\end{equation*}%
or equivalently
\begin{equation*}
L_{k}^{2}=\{u\in L_{loc}^{2}(\mathbb{R})\text{ }{|}\text{ }\exists v\in
L_{\sharp }^{2}\text{ such that }u(x)=v(x)e^{2i\pi kx}\text{ a.e.}\},
\end{equation*}%
where $L_{\sharp }^{2}$ is the traditional notation for $L_{k}^{2}$ in the
periodic case i.e. when $k=0.$ Likewise, for $s\geq 0$ we set
\begin{equation*}
H_{k}^{s}:=L_{k}^{2}\cap H_{loc}^{s}\left(
\mathbb{R}
\right)
\end{equation*}%
bearing in mind that the subscript $\sharp $ would be more appropriate in
the periodic case $k=0.$

\textbf{The modulated two-scale transform} Let us\textbf{\ }assume from now
that the domain $\Omega $ is the union of a finite number of entire cells of
size $\varepsilon $ or equivalently that the sequence $\varepsilon $ is
exactly $\varepsilon _{n}=\frac{\alpha }{n}$ for $n\in
\mathbb{N}
^{\ast }$. Setting $C_{\varepsilon }:=\left\{ {\omega _{\varepsilon
}=\varepsilon l+\varepsilon Y}\text{ }{|}\text{ }{l\in L,\varepsilon
l+\varepsilon Y\subset \Omega }\right\} $ is the set of all cells of $\Omega
$.

\begin{definition}
\label{def_2scale}For any $k\in Y^{\ast }$, the modulated two-scale
transform of the function $u\in L^{2}\left( \Omega \right) $, $%
S_{k}^{\varepsilon }:L^{2}\left( \Omega \right) \rightarrow L^{2}\left( {%
\Omega \times Y}\right) $ is defined by
\begin{equation}
S_{k}^{\varepsilon }u\left( {x,y}\right) =\sum\limits_{\omega _{\varepsilon
}\in C_{\varepsilon }}{u\left( {\varepsilon l_{\omega _{\varepsilon
}}+\varepsilon y}\right) \chi _{\omega _{\varepsilon }}\left( x\right)
e^{-2i\pi kl_{\omega _{\varepsilon }}}},  \label{a4_1}
\end{equation}%
where $\varepsilon l_{\omega _{\varepsilon }}$ stands for the unique node in
$\varepsilon L$ of $\omega _{\varepsilon }$ and ${\chi _{\omega
_{\varepsilon }}}$ is the characteristic function of $\omega _{\varepsilon }$%
.
\end{definition}

The three following properties can be checked by using (\ref{a4_1}) and are
admitted. For $u,v\in L^{2}\left( \Omega \right) $%
\begin{gather}
\left\Vert {S_{k}^{\varepsilon }u}\right\Vert _{L^{2}\left( {\Omega \times Y}%
\right) }^{2}=\int_{\Omega \times Y}{\left\vert {S_{k}^{\varepsilon }u}%
\right\vert ^{2}dxdy}=\sum_{\varepsilon }\int_{\omega _{\varepsilon }}{%
\left\vert u\right\vert ^{2}dx}=\left\Vert u\right\Vert _{_{L^{2}\left(
\Omega \right) }}^{2}\text{,}  \label{Pseudo_sk} \\
S_{k}^{\varepsilon }(uv)=S_{0}^{\varepsilon }(u)S_{k}^{\varepsilon }(v)\text{%
,}  \notag \\
\text{and }S_{k}^{\varepsilon }\left( \partial _{x}u\right) \left( {x,y}%
\right) =\frac{1}{\varepsilon }\partial _{y}S_{k}^{\varepsilon }u\left( {x,y}%
\right) \text{ for }u\in H^{1}\left( \Omega \right) \text{.}  \notag
\end{gather}

\begin{remark}
\label{Sk_conver}Let $k\in Y^{\ast }$\ and a sequence $u^{\varepsilon }$\
bounded in $L^{2}\left( \Omega \right) $\ such that $S_{k}^{\varepsilon
}u^{\varepsilon }$\ converges to $u^{k}$\ in $L^{2}(\Omega \times Y)$\
weakly when $\varepsilon \rightarrow 0$, then $S_{-k}^{\varepsilon
}u^{\varepsilon }$\ converges to some $u^{-k}$\ in $L^{2}(\Omega \times Y)$
weakly. Moreover, since $S_{k}^{\varepsilon }u^{\varepsilon }$\ and $%
S_{-k}^{\varepsilon }u^{\varepsilon }$\ are conjugate then $u^{k}$\ and $%
u^{-k}$\ are also conjugate.
\end{remark}

The adjoint $S_{k}^{\varepsilon \ast }:L^{2}\left( \Omega \times Y\right)
\rightarrow L^{2}\left( \Omega \right) $ of $S_{k}^{\varepsilon }$, is
defined by
\begin{equation}
\int_{\Omega }{\left( {S_{k}^{\varepsilon \ast }v}\right) \left( x\right)
\cdot w\left( x\right) }\text{ }{dx=\int_{\Omega \times Y}{v\left( {x,y}%
\right) \cdot \left( {S_{k}^{\varepsilon }w}\right) \left( {x,y}\right) }}%
\text{ }{dxdy},  \label{Definition_S*}
\end{equation}%
for all $w\in L^{2}\left( \Omega \right) $ and $v\in L^{2}\left( \Omega
\times Y\right) $, and we denote by $\mathfrak{R}$ the operator operating on
functions $v(x,y)$ defined in $\Omega \times \mathbb{R}$,
\begin{equation}
(\mathfrak{R}v)(x)=v(x,\frac{x}{\varepsilon })\text{.}  \label{R}
\end{equation}%
The next Lemma shows that $\mathfrak{R}$ is an approximation of $%
S_{k}^{\varepsilon \ast }$ for $k-$quasi-periodic functions.

\begin{lemma}
\label{1D-Approx R-1}Let $v\in C^{1}\left( \Omega \times Y\right) $ a $k-$%
quasi-periodic function in $y$ then
\begin{equation}
S_{k}^{\varepsilon \ast }v=\mathfrak{R}v+O\left( \varepsilon \right) \quad
\text{in the }L^{2}\left( \Omega \right) \text{ sense}.  \label{S*}
\end{equation}
\end{lemma}

\begin{proof}
The proof is carried out in two steps. First the explicit expression of $%
S_{k}^{\varepsilon \ast }v$ is derived, then the approximation is deduced.

(i) Let us prove that%
\begin{equation*}
({S_{k}^{\varepsilon \ast }}v)\left( x\right) =\sum\limits_{\omega
_{\varepsilon }\in C_{\varepsilon }}{\varepsilon ^{-1}\int_{\omega
_{\varepsilon }}{v\left( {z,\frac{{x-\varepsilon l_{\omega _{\varepsilon }}}%
}{\varepsilon }}\right) }}\text{ }{dz}\text{ }{{\chi _{\omega _{\varepsilon
}}\left( x\right) }}e^{2i\pi kl_{\omega _{\varepsilon }}}.
\end{equation*}%
From the definition of the two-scale transform with $r=\varepsilon l_{\omega
_{\varepsilon }}+\varepsilon y\in \omega _{\varepsilon }$,%
\begin{equation*}
\int_{\Omega \times Y}{v\left( {x,y}\right) \cdot \left( {S_{k}^{\varepsilon
}w}\right) \left( {x,y}\right) dxdy}=\sum\limits_{\omega _{\varepsilon }\in
C_{\varepsilon }}{\int_{\Omega \times \omega _{\varepsilon }}{\varepsilon
^{-1}v\left( {x,\frac{{r-\varepsilon l_{\omega _{\varepsilon }}}}{%
\varepsilon }}\right) \cdot w\left( r\right) {\chi _{\omega _{\varepsilon
}}\left( x\right) }e^{-2i\pi kl_{\omega _{\varepsilon }}}}}\text{ }{dxdr}
\end{equation*}%
or equivalently,
\begin{equation*}
=\int_{\Omega }{\sum\limits_{\omega _{\varepsilon }\in C_{\varepsilon }}{%
\varepsilon ^{-1}\int_{\omega _{\varepsilon }}{v\left( {x,\frac{{%
r-\varepsilon l_{\omega _{\varepsilon }}}}{\varepsilon }}\right) }}}\text{ }{%
{dx\cdot w\left( r\right) }\chi _{\omega _{\varepsilon }}\left( r\right) }%
e^{-2i\pi kl_{\omega _{\varepsilon }}}\text{ }dr.
\end{equation*}%
Changing the variable names and using the definition of ${S_{k}^{\varepsilon
\ast }}$,
\begin{equation*}
\int_{\Omega }{\left( {S_{k}^{\varepsilon \ast }v}\right) \left( x\right)
\cdot w\left( x\right) }\text{ }{dx}=\int_{\Omega }{\sum\limits_{\omega
_{\varepsilon }\in C_{\varepsilon }}{\varepsilon ^{-1}\int_{\omega
_{\varepsilon }}{v\left( {z,\frac{{x-\varepsilon l_{\omega _{\varepsilon }}}%
}{\varepsilon }}\right) dze^{2i\pi k{{{{l_{\omega _{\varepsilon }}}}}}}\cdot
w\left( x\right) }\chi _{\omega _{\varepsilon }}\left( x\right) }}\text{ }dx.
\end{equation*}%
This establishes the explicit expression of ${S_{k}^{\varepsilon \ast }}$.

(ii) Let us derive the expected approximation for $v\in C^{1}\left( \Omega
\times Y\right) $ and $k-$quasi-periodic in $y$.\ Since $\varepsilon
\left\vert Y\right\vert =\left\vert \omega _{\varepsilon }\right\vert $ and
\begin{equation*}
v\left( {z,y}\right) =v\left( {x,y}\right) +\partial _{x}v\left( x,y\right)
\left( z-x\right) +\varepsilon O\left( \varepsilon \right) \text{ in }%
L^{2}(\Omega )\text{ for a.e. }y\in Y
\end{equation*}%
then%
\begin{equation*}
({S_{k}^{\varepsilon \ast }}v)\left( {\varepsilon l_{\omega _{\varepsilon
}}+\varepsilon y}\right) =\frac{1}{\left\vert \omega _{\varepsilon
}\right\vert }\int_{\omega _{\varepsilon }}v\left( {x,y}\right) +\partial
_{x}v\left( x,y\right) \left( z-x\right) \text{ }{dz}\text{ }e^{2i\pi
kl_{\omega _{\varepsilon }}}+O\left( \varepsilon \right)
\end{equation*}%
for a.e. $y\in Y$ and all $\omega _{\varepsilon }\in C_{\varepsilon }$.
Remarking that $z-x=\left( {z-\varepsilon l_{\omega _{\varepsilon }}}\right)
+\left( {\varepsilon l_{\omega _{\varepsilon }}-x}\right) $ and
\begin{equation*}
\int_{\omega _{\varepsilon }}{\left( {z-\varepsilon l_{\omega _{\varepsilon
}}}\right) dz}=-\frac{1}{2}\varepsilon O\left( \varepsilon \right) .
\end{equation*}%
So for all $\omega _{\varepsilon }$ and $y\in Y$,
\begin{equation*}
e^{-2i\pi kl_{\omega _{\varepsilon }}}\left\vert \omega _{\varepsilon
}\right\vert ({S_{k}^{\varepsilon \ast }}v)\left( {\varepsilon l_{\omega
_{\varepsilon }}+\varepsilon y}\right) =\left\vert \omega _{\varepsilon
}\right\vert v\left( {x,y}\right) +(-\frac{1}{2}\varepsilon O\left(
\varepsilon \right) +\left( {\varepsilon ^{2}y}\right) )\partial _{x}v\left(
{x,y}\right) +\varepsilon O\left( \varepsilon \right) .
\end{equation*}%
Therefore,
\begin{equation*}
({S_{k}^{\varepsilon \ast }}v)\left( x\right) =\sum\limits_{\omega
_{\varepsilon }}v\left( {x,\frac{x}{\varepsilon }-l_{\omega _{\varepsilon }}}%
\right) {{{\chi _{\omega _{\varepsilon }}}}}\left( x\right) e^{2i\pi
kl_{\omega _{\varepsilon }}}+O\left( \varepsilon \right) .
\end{equation*}%
Using the $k-$quasi-periodic of $v$ in $y$,%
\begin{equation*}
({S_{k}^{\varepsilon \ast }}v)\left( x\right) =\sum\limits_{\omega
_{\varepsilon }}{v\left( {x,\frac{x}{\varepsilon }}\right) {{\chi _{\omega
_{\varepsilon }}}}\left( x\right) +O\left( \varepsilon \right) }
\end{equation*}%
in $L^{2}(\Omega )$, hence the formula (\ref{S*}) follows.
\end{proof}

In the proof, we constantly use the following consequence.

\begin{corollary}
\label{two-scale-converge}Let $v\in C^{1}\left( \Omega \times Y\right) $ and
$k-$quasi-periodic in $y$, for any sequence $u^{\varepsilon }$ bounded in $%
L^{2}\left( \Omega \right) $ such that $S_{k}^{\varepsilon }u^{\varepsilon }$
converges to $u$ in $L^{2}(\Omega \times Y)$ weakly when $\varepsilon
\rightarrow 0$ then
\begin{equation*}
\int_{\Omega }u^{\varepsilon }\cdot \mathfrak{R}v\text{ }dx\rightarrow
\int_{\Omega \times Y}u\cdot v\text{ }dxdy\text{ \ when }\varepsilon
\rightarrow 0.
\end{equation*}
\end{corollary}

Note that for $k=0$, this corresponds to the definition of two-scale
convergence in \cite{allaire1992homogenization} and\textbf{\ }\cite%
{nguetseng1989general}.

\textbf{Two-scale operators }For a function $v(x,y)$ defined in $\Omega
\times \mathbb{R},$ we pose%
\begin{equation*}
P^{0}v=-\partial _{x}\left( a\partial _{x}v\right) ,\text{ }P^{1}v=-\partial
_{x}\left( a\partial _{y}v\right) -\partial _{y}\left( a\partial
_{x}v\right) \text{ and }P^{2}v=-\partial _{y}\left( a\partial _{y}v\right) ,
\end{equation*}%
so that%
\begin{equation}
P^{\varepsilon }\mathfrak{R}v=\sum_{n=0}^{2}\varepsilon ^{-n}\mathfrak{R}%
P^{n}v.  \label{1D-A R v}
\end{equation}

\textbf{Bloch waves }For a given $k\in Y^{\ast }$, we denote by $(\lambda
_{n}^{k},\phi _{n}^{k})$ the Bloch wave eigenelements indexed by $n\in
\mathbb{N}
^{\ast }$ that are solution to%
\begin{equation}
\mathcal{P}(k):P^{2}\phi _{n}^{k}=\lambda _{n}^{k}\rho \phi _{n}^{k}\text{
in }Y\text{ with }\phi _{n}^{k}\in H_{k}^{2}(Y)\text{ and }\left\Vert \phi
_{n}^{k}\right\Vert _{L^{2}\left( Y\right) }=1.  \label{Blochwaves}
\end{equation}%
The corresponding weak formulation is: find $\phi _{n}^{k}\in H_{k}^{1}(Y)$
solution to
\begin{equation}
\int_{Y}a\partial _{y}\phi _{n}^{k}\cdot \partial _{y}v-\lambda _{n}^{k}\rho
\phi _{n}^{k}\cdot v\text{ }dy=0\text{ for all }v\in H_{k}^{1}(Y).
\label{weak_bloch_waves}
\end{equation}%
Since the operator $P^{2}:H_{k}^{2}(Y)\subset L_{k}^{2}(Y)\rightarrow
L_{k}^{2}(Y)$ is self-adjoint, its spectra is real. Furthermore, for $n,m\in
\mathbb{N}
^{\ast }$, we introduce the coefficients%
\begin{equation}
c(k,n,m)=\int_{Y}a\partial _{y}\phi _{m}^{k}\cdot \phi _{n}^{k}-\phi
_{m}^{k}\cdot a\partial _{y}\phi _{n}^{k}\text{ }dy\text{ and }%
b(k,n,m)=\int_{Y}\rho \phi _{m}^{k}\cdot \phi _{n}^{k}\text{ }dy
\label{def of b and c}
\end{equation}%
and observe that the following properties hold,%
\begin{equation*}
c(k,n,m)=\overline{c(-k,n,m)},\text{ }c(k,m,n)=-\overline{c(k,n,m)},\text{ }%
c(k,n,m)=-c(-k,m,n)
\end{equation*}%
and%
\begin{equation*}
b(k,n,m)=\overline{b(k,m,n)},\text{ }b(k,n,m)=\overline{b(-k,m,n)},\text{ }%
b\left( k,n,n\right) >0.
\end{equation*}%
In particular for $k=0$, if the eigenvectors are chosen as real functions
thus $c\left( 0,n,n\right) =0.$ In the special case $\rho =1$, $b(k,n,m)=1$
for $n=m$ and $b(k,n,m)=0$ otherwise.

\begin{notation}
\label{Conjugate of a Bloch mode}For $k\neq 0$, $\overline{\phi _{n}^{k}}\in
H_{-k}^{2}(Y)$, the conjugate of $\phi _{n}^{k}$, is solution of $\mathcal{P}%
(-k)$. We choose the numbering of eigenvectors $\phi _{n}^{-k}$ so that $%
\phi _{n}^{-k}=\overline{\phi _{n}^{k}}$ and remark that $\lambda
_{n}^{-k}=\lambda _{n}^{k}.$
\end{notation}

\begin{remark}
\label{Bloch eigenvalue property}In one dimension, for $k\in Y^{\ast },$ it
is well-known that all eigenvalue $\lambda _{n}^{k}$ are simple, except for $%
k=0$ where they are double.
\end{remark}

Finally, we denote%
\begin{equation*}
I^{k}=\left\{ k,-k\right\} \text{ if }k\in Y^{\ast }\diagdown \left\{
0\right\} \text{ and }I^{0}=\left\{ 0\right\} \text{ otherwise.}
\end{equation*}

\section{Homogenization of the high-frequency eigenvalue problem \label%
{second}}

For $k\in Y^{\ast }$, we decompose
\begin{equation}
\frac{\alpha k}{\varepsilon }=h_{\varepsilon }^{k}+l_{\varepsilon }^{k}\text{
with }h_{\varepsilon }^{k}=\left[ \frac{\alpha k}{\varepsilon }\right] \text{
and }l_{\varepsilon }^{k}\in \left[ 0,1\right) ,  \label{epsilon_m}
\end{equation}%
and assume that the sequence of the $\varepsilon $ is varying in a set $%
E_{k}\subset \mathbb{R}^{+\ast }$ depending on $k$ so that%
\begin{equation}
l_{\varepsilon }^{k}\rightarrow l^{k}\text{ when }\varepsilon \rightarrow 0%
\text{ and }\varepsilon \in E_{k}\text{ with }l^{k}\in \left[ 0,1\right) .
\label{l}
\end{equation}%
We note that for $k=0$, $h_{\varepsilon }^{k}=0,$ $l_{\varepsilon }^{k}=0$,
so $l^{k}=0$ and $E_{0}=\mathbb{R}^{+\ast }$.

\subsection{Main result \label{result}}

The macroscopic equation is stated for each $k\in Y^{\ast }$ and each Bloch
wave eigenvalue $\lambda _{n}^{k}$. For $k\neq 0$, we assume that $c\left(
\sigma ,n,n\right) \neq 0$ for each $\sigma \in I^{k}$, so it is stated as
an eigenvalue problem%
\begin{equation}
c\left( \sigma ,n,n\right) {\partial _{x}u_{n}^{\sigma }+\lambda }%
^{1}b\left( \sigma ,n,n\right) {u_{n}^{\sigma }}=0\text{ \ in \ }\Omega
\label{marco_1}
\end{equation}%
for each $\sigma $, with the boundary conditions%
\begin{equation}
\sum_{\sigma \in I^{k}}{u_{n}^{\sigma }\left( x\right) \phi _{n}^{\sigma
}\left( 0\right) e^{sign\left( \sigma \right) 2i\pi \frac{l^{k}{x}}{\alpha }}%
}=0\text{ on }x\in \partial \Omega ,  \label{boundary_macro_1}
\end{equation}%
where $l^{k}$ is defined in (\ref{l}). We observe that the first order
operator $c\left( k,n,n\right) \left(
\begin{array}{cc}
\partial _{x} & 0 \\
0 & -\partial _{x}%
\end{array}%
\right) $ of this system is self-adjoint on the domain%
\begin{equation*}
D^{k}=\left\{ \left( u_{n},v_{n}\right) \in H^{1}\left( \Omega \right) ^{2}%
\text{ satisfying (\ref{boundary_macro_1})}\right\}
\end{equation*}%
so $\lambda ^{1}$ is real.

For $k=0$, assuming that $\lambda _{n}^{0}$ is a double eigenvalue
corresponding to two eigenvectors $\phi _{n}^{0}$ and $\phi _{m}^{0}$, and
that $c\left( 0,n,m\right) \neq 0$, the macroscopic system states%
\begin{equation}
\sum\limits_{q\in \left\{ n,m\right\} }c\left( 0,p,q\right) {\partial
_{x}u_{q}^{0}+\lambda }^{1}b\left( 0,p,q\right) {u_{q}^{0}}=0\text{ in\ }%
\Omega \text{ for }p\in \left\{ n,m\right\} ,  \label{marco_1_0}
\end{equation}%
with the boundary conditions%
\begin{equation}
\sum\limits_{q\in \left\{ n,m\right\} }{u_{q}^{0}\left( x\right) \phi
_{q}^{0}\left( 0\right) }=0\text{ on }x\in \partial \Omega \text{.}
\label{boundary_macro_1_0}
\end{equation}%
Again $\lambda ^{1}\in \mathbb{R}$ since $c\left( 0,n,m\right) \left(
\begin{array}{cc}
0 & \partial _{x} \\
-\partial _{x} & 0%
\end{array}%
\right) $ is self-adjoint on%
\begin{equation*}
D^{0}=\left\{ \left( u_{n},u_{m}\right) \in H^{1}\left( \Omega \right) ^{2}%
\text{ satisfying (\ref{boundary_macro_1_0})}\right\} \text{.}
\end{equation*}

\begin{remark}
(i) If $c\left( k,n,n\right) =0$ for $k\neq 0$ or $c\left( 0,p,q\right) =0$
for all $p,q$ varying in $\left\{ n,m\right\} ,$ the macroscopic equations (%
\ref{marco_1}) or (\ref{marco_1_0}) are $\lambda ^{1}=0$ or $u=\left(
u_{n}^{\sigma }\right) _{n,\sigma }=0$. But $u=0$ is impossible since $%
\left\Vert w^{\varepsilon }\right\Vert _{L^{2}\left( \Omega \right) }=1$ for
all eigenmodes $w^{\varepsilon }$. So $\lambda ^{1}=0$ and this model does
not provide any equation for $u_{n}^{\sigma }$.

(ii) For $k\neq 0$, if ${\phi _{m}^{k}}\left( 0\right) =0$ then ${\phi
_{m}^{k}}\left( 1\right) =0$ and $\phi _{m}^{k}$ is a periodic solution that
is a solution of $k=0$. So, we consider always that $\phi _{m}^{k}\left(
0\right) \neq 0$ for $k\neq 0$.

(iii) For $k=0$, in case where $\phi _{n}(0)=\phi _{m}(0)=0$ the boundary
conditions of the macroscopic equation vanishes.
\end{remark}

\begin{remark}
\label{remark_Bloch_wave}This work focuses on the Bloch spectrum.\ To avoid
eigenmodes related to the boundary spectrum, according to Proposition 7.7 in
\cite{allaire1998bloch} we shall assume that the weak limit of $%
S_{k}^{\varepsilon }w^{\varepsilon }$\ in $L^{2}\left( \Omega
;H^{1}(Y)\right) $ is not vanishing.
\end{remark}

The main Theorem states as follows.

\begin{theorem}
\label{a7_71} For $k\in Y^{\ast },$ let $\left( \lambda ^{\varepsilon
},w^{\varepsilon }\right) $ be solution of (\ref{a2_1}) then $%
\sum\limits_{\sigma \in I^{k}}S_{\sigma }^{\varepsilon }w^{\varepsilon }$ is
bounded in $L^{2}\left( \Omega ;H^{1}(Y)\right) $. For $\varepsilon \in
E_{k} $, as in (\ref{epsilon_m}, \ref{l}), assuming that the weak limit of $%
S_{k}^{\varepsilon }w^{\varepsilon }$ in $L^{2}\left( \Omega
;H^{1}(Y)\right) $\ is non-vanishing and the renormalized sequence $%
\varepsilon ^{2}\lambda ^{\varepsilon }$\ satisfies the decomposition (\ref%
{eigenvalue_decomposition}), there exists $n\in
\mathbb{N}
^{\ast }$\ such that $\lambda ^{0}=\lambda _{n}^{k}$\ with $\lambda _{n}^{k}$%
\ an\ eigenvalue of the Bloch wave spectrum and the limit $g_{k}$ of any
weakly converging extracted subsequence of $\sum\limits_{\sigma \in
I^{k}}S_{\sigma }^{\varepsilon }w^{\varepsilon }$ in $L^{2}\left( \Omega
;H^{1}(Y)\right) $ can be decomposed on the Bloch modes%
\begin{equation}
g_{k}\left( {x,y}\right) =\sum_{\sigma \in I^{k}}{u_{n}^{\sigma }\left(
x\right) \phi _{n}^{\sigma }\left( y\right) }\text{ for }k\neq 0\text{ and }%
g_{0}\left( {x,y}\right) =\sum_{q\in \{n,m\}}{u_{q}^{0}\left( x\right) \phi
_{q}^{0}\left( y\right) }\text{ otherwise}  \label{appro_theorem}
\end{equation}%
Moreover, $u_{m}^{\sigma }\in H^{1}(\Omega )$ and $\left( u_{m}^{\sigma
}\right) _{m,\sigma }$ are solutions of the macroscopic equations (\ref%
{marco_1}, \ref{boundary_macro_1}) and (\ref{marco_1_0}, \ref%
{boundary_macro_1_0}). Finally, $u_{m}^{k}$ and $u_{m}^{-k}$ are conjugate.
\end{theorem}

Thus, it follows from (\ref{appro_theorem}) that the physical solution $%
w^{\varepsilon }$ is approximated by two-scale modes
\begin{equation}
w^{\varepsilon }\left( x\right) \approx \sum_{\sigma \in I^{k}}{%
u_{n}^{\sigma }\left( x\right) \phi _{n}^{\sigma }\left( {\frac{x}{%
\varepsilon }}\right) }\text{ for }k\neq 0\text{ and }w^{\varepsilon }\left(
x\right) \approx \sum_{q\in \{n,m\}}{u_{q}^{0}\left( x\right) \phi
_{q}^{0}\left( {\frac{x}{\varepsilon }}\right) }\text{ otherwise.}
\label{Physic_approximation}
\end{equation}

The boundary conditions (\ref{boundary_macro_1}) and (\ref%
{boundary_macro_1_0}) can be directly derived by replacing $w^{\varepsilon }$
in the physical boundary condition by its approximations,
\begin{equation}
\sum_{\sigma \in I^{k}}{u_{n}^{\sigma }\left( x\right) \phi _{n}^{\sigma
}\left( {\frac{x}{\varepsilon }}\right) =0}\text{ for }k\neq 0\text{ and }%
\sum_{q\in \{n,m\}}{u_{q}^{0}\left( x\right) \phi _{q}^{0}\left( {\frac{x}{%
\varepsilon }}\right) =0}\text{ otherwise at }x\in \partial \Omega \text{.}
\label{idea_boundary}
\end{equation}%
For $k\neq 0$, they result from%
\begin{equation*}
{\phi _{n}^{\sigma }\left( {\frac{x}{\varepsilon }}\right) =\phi
_{n}^{\sigma }\left( {0}\right) e}^{2i\pi \sigma \frac{x}{\varepsilon }}={%
\phi _{n}^{\sigma }\left( {0}\right) e}^{sign\left( \sigma \right) 2i\pi x%
\frac{h_{\varepsilon }^{k}+l_{\varepsilon }^{k}}{\alpha }}={\phi
_{n}^{\sigma }\left( {0}\right) e}^{sign\left( \sigma \right) 2i\pi x\frac{%
l_{\varepsilon }^{k}}{\alpha }}\text{ for }x\in \partial \Omega
\end{equation*}%
and the assumption $l_{\varepsilon }^{k}\rightarrow l^{k}$. For $k=0$, the
conditions follow from the periodicity of ${\phi _{n}^{0}}$. Furthermore, we
observe that $g_{k}\left( x,0\right) $ and $g_{k}\left( x,1\right) $ are
generally not vanishing except for $k=0$.

\begin{proposition}
\label{solution}For $k\in Y^{\ast }$, $n\in
\mathbb{N}
^{\ast }$, if the macroscopic solution $u_{n}^{k}$\ is a non-vanishing
constant, then any two-scale mode (\ref{Physic_approximation}) is a physical
eigenmode i.e. a solution to (\ref{a2_1}).
\end{proposition}

\begin{proof}
For $k\in Y^{\ast }$, $n\in
\mathbb{N}
^{\ast }$, if the macroscopic solution $u_{n}^{k}$ is constant then $\lambda
^{1}=0$ and $\left( u_{m}^{\sigma }\right) _{m,\sigma }$ are constant for
all $\sigma \in I^{k}$ and $m\in
\mathbb{N}
^{\ast }$ such that $\lambda _{m}^{\sigma }=\lambda _{n}^{\sigma }$. Now, we
consider $\rho =1$ and the proof is similar for $\rho \neq 1$\textbf{. }%
Based on Remark \ref{interval for lambda^1} about the macroscopic solutions
in Section \ref{analytic_solution}, $\lambda ^{1}=0\,$is equivalent to $\ell
=\frac{2k\alpha }{\varepsilon }$. From the $\sigma -$quasi-periodicity of $%
\phi _{n}^{\sigma }$,%
\begin{equation*}
\phi _{n}^{\sigma }\left( \frac{\alpha }{\varepsilon }\right) =\phi
_{n}^{\sigma }\left( 0\right) e^{sign\left( \sigma \right) 2i\pi k\frac{%
\alpha }{\varepsilon }}=\phi _{n}^{\sigma }\left( 0\right) e^{sign\left(
\sigma \right) i\pi \ell }=\pm \phi _{n}^{\sigma }\left( 0\right) ,
\end{equation*}%
then $\phi _{n}^{\sigma }$ is $\alpha -$periodic or $\alpha -$anti-periodic
for $\sigma \in I^{k}$. Hence $\phi _{n}^{\sigma }\left( \frac{x}{%
\varepsilon }\right) $ is a solution of the equation%
\begin{gather}
\partial _{x}\left( a\left( \frac{x}{\varepsilon }\right) \partial _{x}\phi
_{n}^{\sigma }\left( \frac{x}{\varepsilon }\right) \right) =-\frac{\lambda
_{n}^{\sigma }}{\varepsilon ^{2}}\phi _{n}^{\sigma }\left( \frac{x}{%
\varepsilon }\right) \text{ in }\Omega \text{ }  \label{phy-Bloch-wave} \\
\text{and }\phi _{n}^{\sigma }\left( \frac{x}{\varepsilon }\right) \text{ is
}\alpha -\text{periodic or }\alpha -\text{anti-periodic,}  \notag
\end{gather}%
and $u_{m}^{\sigma }\phi _{m}^{\sigma }\left( \frac{x}{\varepsilon }\right) $
is also a solution of (\ref{phy-Bloch-wave}). Denote by $w^{\varepsilon
}:=\sum\limits_{\sigma \in I^{k}}\sum\limits_{m}u_{m}^{\sigma }\phi
_{m}^{\sigma }\left( \frac{x}{\varepsilon }\right) $ and observe that $%
w^{\varepsilon }$ is a solution of the equation%
\begin{equation*}
\partial _{x}\left( a^{\varepsilon }\partial _{x}w^{\varepsilon }\right)
=-\lambda ^{\varepsilon }w^{\varepsilon }\text{ in }\Omega
\end{equation*}%
with the boundary conditions%
\begin{equation*}
w^{\varepsilon }\left( 0\right) =\sum\limits_{\sigma \in
I^{k}}\sum\limits_{m}u_{m}^{\sigma }\phi _{m}^{\sigma }\left( 0\right) =0%
\text{ and }w^{\varepsilon }\left( \alpha \right) =\sum\limits_{\sigma \in
I^{k}}\sum\limits_{m}u_{m}^{\sigma }\phi _{m}^{\sigma }\left( \frac{x}{%
\varepsilon }\right) =\pm w^{\varepsilon }\left( 0\right) =0.
\end{equation*}%
Finally, Proposition \ref{solution} is concluded.
\end{proof}

\begin{remark}
\label{exact_solution}The converse is probably true, and is numerically
studied in Section \ref{problem1}, i.e. for any\textbf{\ }$\left( \lambda
^{\varepsilon },w^{\varepsilon }\right) $\ solution to (\ref{a2_1}), there
exist $k\in Y^{\ast }$, $n\in
\mathbb{N}
^{\ast }$ and two complex numbers $\xi _{1}$\ and $\xi _{2}$ such that $%
\lambda ^{\varepsilon }=\lambda _{n}^{k}/\varepsilon ^{2}$ and%
\begin{equation}
w^{\varepsilon }\left( x\right) =\xi _{1}\phi _{n}^{k}\left( \frac{x}{%
\varepsilon }\right) +\xi _{2}\phi _{n}^{-k}\left( \frac{x}{\varepsilon }%
\right) \text{ if }k\neq 0\text{ and }w^{\varepsilon }\left( x\right) =\xi
_{1}\phi _{n}^{0}\left( \frac{x}{\varepsilon }\right) +\xi _{2}\phi
_{m}^{0}\left( \frac{x}{\varepsilon }\right) \text{ otherwise}
\label{exact_solution_decompose}
\end{equation}%
for $\xi _{1},\xi _{2}$ two numbers such that the boundary conditions (\ref%
{boundary_macro_1_0}), respectively (\ref{boundary_macro_1}), are satisfied
for $k=0,$\ respectively for $k\neq 0$. In the later case $\xi _{1}$ and $%
\xi _{2}$ are conjugate.
\end{remark}

\begin{remark}
\label{Application to the wave equation}(i) The case of non-constant
coefficients $u_{n}^{k}$ is used for approximations of the solution to the
homogenized wave equation that may be derived from \cite{brassart2010two}.
In such case $k$ belongs to a finite subset $L_{K}^{\ast }$ of $Y^{\ast }$
made with values distant from $1/K$ and including $0$. We cannot expect that
there always exists a pair $(k,n)$ such that $u_{n}^{k}$\ is a constant.

(ii) The case of non-constant coefficients $u_{n}^{k}$ is also seen as a
preparation to derive homogenized spectral problems in higher dimension
where the boundary conditions constitute a more difficult problem and may
require a more general solution than constant $u_{n}^{k}$.
\end{remark}

Proof of Theorem \ref{a7_71}

\begin{proof}
The proof is based on Lemma \ref{Lemme-decomposition-Bloch} in Section \ref%
{First order equation} and on the macroscopic model derivation in Section %
\ref{The zeros equation}. For a given $k\in Y^{\ast }$, let $w^{\varepsilon
} $ be solution of \ (\ref{a2_1}) which is bounded in $L^{2}(\Omega )$, the
property (\ref{Pseudo_sk}) yields the uniform boundness of $\left\Vert
S_{\sigma }^{\varepsilon }w^{\varepsilon }\right\Vert _{L^{2}\left( \Omega
\times Y\right) }$ for any $\sigma \in I^{k}$. So there exist $w^{\sigma
}\in $ $L^{2}(\Omega \times Y)$ such that up the extraction of a subsequence
$S_{\sigma }^{\varepsilon }w^{\varepsilon }\rightarrow w^{\sigma }$ in $%
L^{2}\left( \Omega \times Y\right) $ weakly. Since $\left\Vert
S_{\sigma }^{\varepsilon }\left( \varepsilon \partial
_{x}w^{\varepsilon }\right) \right\Vert _{L^{2}\left( \Omega \times
Y\right) }=\left\Vert
\partial _{y}S_{\sigma }^{\varepsilon }w^{\varepsilon }\right\Vert
_{L^{2}\left( \Omega \times Y\right) }$ is uniformly bounded as
$\left\Vert \varepsilon \partial _{x}w^{\varepsilon }\right\Vert
_{L^{2}\left( \Omega \right) }$. Hence
\begin{equation*}
\lim_{\varepsilon \rightarrow 0}\int_{\Omega \times Y}\partial _{y}S_{\sigma
}^{\varepsilon }w^{\varepsilon }\cdot vdxdy=\lim_{\varepsilon \rightarrow
0}\int_{\Omega \times Y}-S_{\sigma }^{\varepsilon }w^{\varepsilon }\cdot
\partial _{y}vdxdy=-\int_{\Omega \times Y}w^{\sigma }\cdot \partial _{y}vdxdy
\end{equation*}
for all $v\in L^{2}(\Omega ;H_{0}^{1}(Y))$, and $w^{\sigma }\in
L^{2}(\Omega
;H^{1}(Y))$ then%
\begin{equation*}
\lim_{\varepsilon \rightarrow 0}\int_{\Omega \times Y}\partial _{y}S_{\sigma
}^{\varepsilon }w^{\varepsilon }\cdot vdxdy=\int_{\Omega \times Y}\partial
_{y}w^{\sigma }\cdot vdxdy.
\end{equation*}%
Therefore $S_{\sigma }^{\varepsilon }w^{\varepsilon }$ tends weakly to $%
w^{\sigma }$ also in $L^{2}(\Omega ;H^{1}\left( Y\right) )$. Hence, $%
\sum\limits_{\sigma \in I^{k}}S_{\sigma }^{\varepsilon }w^{\varepsilon }$
converges to
\begin{equation*}
g_{k}\left( x,y\right) =\sum\limits_{\sigma \in I^{k}}w^{\sigma }\left(
x,y\right) .
\end{equation*}%
Using the decomposition (\ref{decompose_bloch_mode}) of $w^{\sigma }$ in
Lemma \ref{Lemme-decomposition-Bloch}, for $\left( \phi _{p}^{\sigma
}\right) _{\sigma ,p}$ the Bloch wave eigenmodes corresponding to $\lambda
^{0}$,%
\begin{equation*}
\left\{
\begin{array}{l}
g_{k}\left( x,y\right) =\sum\limits_{\sigma \in I^{k}}u_{n}^{\sigma }\left(
x\right) \phi _{n}^{\sigma }\left( y\right) \text{ for }k\neq 0, \\
g_{0}\left( x,y\right) =\sum\limits_{p\in \left\{ n,m\right\}
}u_{p}^{0}\left( x\right) \phi _{p}^{0}\left( y\right) \text{ for }k=0.%
\end{array}%
\right.
\end{equation*}%
Finally, $\left( u_{p}^{\sigma }\right) _{\sigma ,p}$ is solution of the
macroscopic problem as proved in Section \ref{The zeros equation}.
\end{proof}

\subsection{Modal decomposition on the Bloch modes \label{First order
equation}}

\begin{lemma}
\label{Lemme-decomposition-Bloch}For $\left( \lambda ^{\varepsilon
},w^{\varepsilon }\right) $ solution of (\ref{a2_1}) and satisfying (\ref%
{eigenvalue_decomposition}), for a fixed $k\in Y^{\ast }$ there exists at
least a subsequence of $S_{k}^{\varepsilon }w^{\varepsilon }$ converging
weakly towards\textbf{\ }non-vanishing function $w^{k}$ in $L^{2}\left(
\Omega \times Y\right) $ when $\varepsilon $ tends to zero. If $w^{k}\in
L^{2}(\Omega ;H^{2}(Y))$ then $\left( \lambda ^{0},w^{k}\right) $ is
solution of the Bloch wave equation (\ref{Blochwaves}) and $w^{k}$ admits
the modal decomposition,%
\begin{equation}
w^{k}\left( x,y\right) =\sum\limits_{m}u_{m}^{k}\left( x\right) \phi
_{m}^{k}\left( y\right) \text{ for }u_{m}^{k}\in L^{2}\left( \Omega \right)
\label{decompose_bloch_mode}
\end{equation}%
where the sum is over all Bloch modes $\phi _{m}^{k}$ associated to $\lambda
^{0}$. Moreover for $k\neq 0$ the two factors $u_{m}^{k}$ and $u_{m}^{-k}$
are conjugate.
\end{lemma}

\begin{proof}
The test functions of the weak formulation (\ref{1D-weakformulation1}) are
chosen as%
\begin{equation}
v^{\varepsilon }:=\mathfrak{R}v\in H_{0}^{1}(\Omega )\cap H^{2}(\Omega ),
\label{1D-test function}
\end{equation}%
with%
\begin{equation}
v\in H_{0}^{1}(\Omega ;L_{k}^{2}(Y))\cap L^{2}(\Omega ;H_{k}^{2}(Y))\cap
H^{2}\left( \Omega ;L_{k}^{2}\left( Y\right) \right) .
\label{test function v}
\end{equation}%
Applying two integrations by parts and the boundary conditions satisfied by $%
w^{\varepsilon }$ and by $\mathfrak{R}v$, it remains%
\begin{equation}
\int_{\Omega }w^{\varepsilon }\cdot (P^{\varepsilon }-\lambda ^{\varepsilon
}\rho ^{\varepsilon })v^{\varepsilon }\text{ }dx=0.
\label{Very weak formulation}
\end{equation}%
From (\ref{1D-A R v}) multiplied by $\varepsilon ^{2}$ and (\ref%
{eigenvalue_decomposition}),%
\begin{equation*}
\int_{\Omega }w^{\varepsilon }\cdot \mathfrak{R}((P^{2}-\lambda ^{0}\rho )v)%
\text{ }dx=O(\varepsilon )\text{.}
\end{equation*}%
Since $(P^{2}-\lambda ^{0}\rho )v$ is $k-$quasi-periodic and $%
S_{k}^{\varepsilon }w^{\varepsilon }\rightarrow w^{k}$ in $L^{2}(\Omega
\times Y)$ weakly, Corollary \ref{two-scale-converge} allows to pass to the
limit%
\begin{equation*}
\int_{\Omega \times Y}w^{k}\cdot (P^{2}-\lambda ^{0}\rho )v\text{ }dxdy=0,
\end{equation*}%
or equivalently%
\begin{equation}
\int_{\Omega \times Y}w^{k}\cdot \partial _{y}\left( a\partial _{y}v\right)
+w^{k}\cdot \lambda ^{0}\rho v\text{ }dxdy=0.
\label{1D-second order eq in very
weak form}
\end{equation}%
Using the assumption $w^{k}\in L^{2}(\Omega ;H^{2}(Y))$ and applying
integrations by parts,%
\begin{equation*}
\int_{\Omega \times Y}\partial _{y}\left( a\partial _{y}w^{k}\right) \cdot
v+w^{k}\cdot \lambda ^{0}\rho v\text{ }dxdy+\int_{\Omega \times \partial
Y}aw^{k}\cdot \partial _{y}v-a\partial _{y}w^{k}\cdot v\,dxdy=0.
\end{equation*}%
Then, choosing test functions $v\in L^{2}(\Omega ;H_{0}^{2}(Y))$ comes the
strong form%
\begin{equation}
-\partial _{y}\left( a\partial _{y}w^{k}\right) =\lambda ^{0}\rho w^{k}\text{
in }\Omega \times Y.  \label{1D-second order equation in strong form}
\end{equation}%
So, it remains%
\begin{equation*}
\int_{\Omega }\left[ aw^{k}\cdot \partial _{y}v-a\partial _{y}w^{k}\cdot v%
\right] _{0}^{1}\text{ }dx=0
\end{equation*}%
for general test functions (\ref{test function v}), which implies that $%
w^{k} $ and $\partial _{y}w^{k}$ are $k-$quasi-periodic in the variable $y$.

As we know that $\lambda ^{0}$ is an eigenvalue $\lambda _{n}^{k}$ of the
Bloch wave spectrum, then $w^{k}$ is a Bloch eigenvector and is decomposed
as
\begin{equation*}
w^{k}\left( x,y\right) =\sum\limits_{m}u_{m}^{k}\left( x\right) \phi
_{m}^{k}\left( y\right) \text{ with }u_{m}^{k}\in L^{2}\left( \Omega \right)
\end{equation*}%
the sum being over all Bloch modes $\phi _{m}^{k}$ associated to $\lambda
^{0}$ where $u_{m}^{k}(x)=\int_{Y}w^{k}(x,y)\cdot \phi _{m}^{k}(y)$ $dy$.
For $k\neq 0$, ${\phi _{m}^{k}=}\overline{{\phi _{m}^{-k}}}$ and from
Definition \ref{def_2scale} of modulated two-scale transform, $%
S_{k}^{\varepsilon }w^{\varepsilon }=\overline{S_{-k}^{\varepsilon
}w^{\varepsilon }}$ thus $u_{m}^{k}$ and $u_{m}^{-k}$ are conjugate i.e. $%
u_{m}^{k}=\overline{u_{m}^{-k}}$.
\end{proof}

\subsection{Derivation of the macroscopic equation \label{The zeros equation}%
}

In the macroscopic model derivation, we distinguish between the two cases $%
k\neq 0$ and $k=0$.

\subsubsection{Case $k\neq 0$ \label{k_non_0_single_value}}

We consider $\lambda ^{0}=\lambda _{n}^{k}$ and the two conjugate
eigenvectors $\phi _{n}^{k}$ and $\phi _{n}^{-k}$ discussed in Notation \ref%
{Conjugate of a Bloch mode}. We restart from the very weak formulation (\ref%
{Very weak formulation}) with the test function
\begin{equation}
v^{\varepsilon }(x):=\mathfrak{R}(v^{k}+v^{-k})\in H_{0}^{1}\left( \Omega
\right) \cap H^{2}\left( \Omega \right) .  \label{test_function_marc}
\end{equation}%
Furthermore, we pose $v^{\sigma }(x,y)=\psi ^{\sigma }(x)\phi _{n}^{\sigma
}(y)$ with $\psi ^{\sigma }\in H^{2}(\Omega )$ for $\sigma \in I^{k}$ and
use the $\sigma -$quasi-periodicity of $\phi _{n}^{\sigma },$ i.e. $\phi
_{n}^{\sigma }\left( \frac{x}{\varepsilon }\right) =\phi _{n}^{\sigma
}\left( 0\right) e^{2i\pi k\frac{x}{\varepsilon }}$ at any $x\in \partial
\Omega $.\ So the boundary condition in (\ref{test_function_marc}) is
equivalent to

\begin{equation*}
\psi ^{k}\left( x\right) \phi _{n}^{k}(0)e^{2i\pi k\frac{x}{\varepsilon }%
}+\psi ^{-k}\left( x\right) \phi _{n}^{-k}(0)e^{-2i\pi k\frac{x}{\varepsilon
}}=0\text{ at any }x\in \partial \Omega .
\end{equation*}%
Applying the relation (\ref{epsilon_m}),%
\begin{equation*}
\psi ^{k}(x)\phi _{n}^{k}(0)e^{2i\pi x\frac{h_{\varepsilon
}^{k}+l_{\varepsilon }^{k}}{\alpha }}+\psi ^{-k}(x)\phi
_{n}^{-k}(0)e^{-2i\pi x\frac{h_{\varepsilon }^{k}+l_{\varepsilon }^{k}}{%
\alpha }}=0.
\end{equation*}%
Since $x\frac{h_{\varepsilon }^{k}}{\alpha }=0$ at $x=0$ and $x\frac{%
h_{\varepsilon }^{k}}{\alpha }=h_{\varepsilon }^{k}$ at $x=\alpha $ with $%
h_{\varepsilon }^{k}\in
\mathbb{Z}
$ then $e^{\pm 2i\pi x\frac{h_{\varepsilon }^{k}}{\alpha }}=1$. From (\ref{l}%
), $e^{\pm 2i\pi \frac{l_{\varepsilon }^{k}x}{\alpha }}\rightarrow e^{\pm
2i\pi \frac{l^{k}x}{\alpha }}$ when $\varepsilon \rightarrow 0$. Passing to
the limit, the boundary conditions of the test function are%
\begin{equation}
\psi ^{k}(x)\phi _{n}^{k}(0)e^{2i\pi \frac{l^{k}x}{\alpha }}+\psi
^{-k}(x)\phi _{n}^{-k}(0)e^{-2i\pi \frac{l^{k}x}{\alpha }}=0\text{ on }%
\partial \Omega .  \label{boundary condition on test function}
\end{equation}%
From (\ref{1D-A R v}) multiplied by $\varepsilon $, (\ref%
{eigenvalue_decomposition}) and $P^{2}v^{\sigma }-\lambda ^{0}\rho v^{\sigma
}=0$,%
\begin{equation}
\sum_{\sigma \in I^{k}}\int_{\Omega }w^{\varepsilon }\cdot \mathfrak{R}%
(-P^{1}v^{\sigma }+\lambda ^{1}\rho v^{\sigma })\text{ }dx=O(\varepsilon )%
\text{.}  \label{1D-weakcomplete}
\end{equation}%
Extracting a subsequence of $w^{\varepsilon }$ so that $S_{k}^{\varepsilon
}w^{\varepsilon }$ and $S_{-k}^{\varepsilon }w^{\varepsilon }$ are
converging to $w^{k}$ and $w^{-k}$ in $L^{2}(\Omega \times Y)$ weak, since $%
-P^{1}v^{\sigma }+\lambda ^{1}\rho v^{\sigma }$ is $\sigma -$quasi-periodic
then Corollary \ref{two-scale-converge} infers that%
\begin{equation*}
\sum_{\sigma \in I^{k}}\int_{\Omega \times Y}w^{\sigma }\cdot
(-P^{1}v^{\sigma }+\lambda ^{1}\rho v^{\sigma })\text{ }dxdy=0,
\end{equation*}%
i.e.%
\begin{equation*}
\sum_{\sigma \in I^{k}}\int_{\Omega \times Y}w^{\sigma }\cdot \left(
\partial _{x}\left( a\partial _{y}v^{\sigma }\right) +\partial _{y}\left(
a\partial _{x}v^{\sigma }\right) +\lambda ^{1}\rho v^{\sigma }\text{ }%
\right) dxdy=0.
\end{equation*}%
This is the very weak form of the macroscopic equation for all test
functions $v^{\sigma }\in H^{1}\left( \Omega ;H_{k}^{1}\left( Y\right)
\right) $, reached by density, satisfying (\ref{boundary condition on test
function}). Now, we derive the strong formulation. We assume that $w^{\sigma
}\in H^{1}(\Omega ;L^{2}(Y))$, since $w^{\sigma }\in L^{2}(\Omega ;H^{1}(Y))$
after two integrations by parts,%
\begin{gather*}
\sum_{\sigma \in I^{k}}\left[ \int_{\Omega \times Y}\partial _{y}\left(
a\partial _{x}w^{\sigma }\right) \cdot v^{\sigma }+\partial _{x}\left(
a\partial _{y}w^{\sigma }\right) \cdot v^{\sigma }+\lambda ^{1}\rho
w^{\sigma }\cdot v^{\sigma }\text{ }dxdy\right. \\
+\left. \int_{\partial \Omega \times Y}w^{\sigma }\cdot a\partial
_{y}v^{\sigma }-a\partial _{y}w^{\sigma }\cdot v^{\sigma }\text{ }dxdy\right.
\\
+\left. \int_{\Omega \times \partial Y}w^{\sigma }\cdot a\partial
_{x}v^{\sigma }-a\partial _{x}w^{\sigma }\cdot v^{\sigma }\text{ }dxdy\right]
=0.
\end{gather*}%
From Lemma \ref{Lemme-decomposition-Bloch}, $w^{\sigma }$ is solution to the
Bloch mode equation and is decomposed as%
\begin{equation}
w^{\sigma }(x,y)=u^{\sigma }(x)\phi _{n}^{\sigma }(y)\text{.}
\label{1D-decompose-k}
\end{equation}%
After replacement,%
\begin{gather}
\sum_{\sigma }\left[ \int_{Y}\partial _{y}(a\phi _{n}^{\sigma })\cdot \phi
_{n}^{\sigma }+a\partial _{y}\phi _{n}^{\sigma }\cdot \phi _{n}^{\sigma }%
\text{ }dy\int_{\Omega }\partial _{x}u^{\sigma }\cdot \psi ^{\sigma
}dx+\lambda ^{1}\int_{Y}\rho \phi _{n}^{\sigma }\cdot \phi _{n}^{\sigma }%
\text{ }dy\int_{\Omega }u^{\sigma }\cdot \psi ^{\sigma }dx\right.
\label{step} \\
+\int_{Y}\phi _{n}^{\sigma }\cdot a\partial _{y}\phi _{n}^{\sigma
}-a\partial _{y}\phi _{n}^{\sigma }\cdot \phi _{n}^{\sigma }\text{ }%
dy\int_{\partial \Omega }u^{\sigma }\cdot \psi ^{\sigma }\text{ }dx  \notag
\\
+\left. \int_{\partial Y}\phi _{n}^{\sigma }\cdot a\phi _{n}^{\sigma }\text{
}dy\int_{\Omega }u^{\sigma }\cdot \partial _{x}\psi ^{\sigma }-\partial
_{x}u^{\sigma }\cdot \psi ^{\sigma }\text{ }dx\right] =0.  \notag
\end{gather}%
Let us recall that $b(.,.,.)$ and $c(.,.,.)$ have been defined in (\ref{def
of b and c}). For the sake of simplicity, we use $c(\sigma ,n):=c(\sigma
,n,n)$ and $b(\sigma ,n):=b(\sigma ,n,n)$ and observe that%
\begin{equation*}
\int_{Y}\partial _{y}(a\phi _{n}^{\sigma })\cdot \phi _{n}^{\sigma
}+a\partial _{y}\phi _{n}^{\sigma }\cdot \phi _{n}^{\sigma }\text{ }%
dy=c(\sigma ,n),
\end{equation*}%
which results from integrations by parts and from the $\sigma -$%
quasi-periodicity of $\phi _{n}^{\sigma }$. So, using the $\sigma $%
-quasi-periodicity of $\phi _{n}^{\sigma }$, (\ref{step}) can be rewritten as%
\begin{equation*}
\sum_{\sigma }\left[ \int_{\Omega }(c(\sigma ,n)\partial _{x}u^{\sigma
}+\lambda ^{1}b\left( \sigma ,n\right) u^{\sigma })\cdot \psi ^{\sigma }%
\text{ }dx-c(\sigma ,n)\int_{\partial \Omega }u^{\sigma }\cdot \psi ^{\sigma
}\text{ }dx\right] =0.
\end{equation*}%
Choosing the test function $\psi ^{\sigma }=0$ on $\partial \Omega $, the
boundary condition (\ref{boundary condition on test function}) is satisfied
and by density of $H_{0}^{1}\left( \Omega \right) $ in $L^{2}\left( \Omega
\right) ,$ the internal equation satisfied by $u^{\sigma }$ follows,%
\begin{equation}
c(\sigma ,n)\partial _{x}u^{\sigma }+\lambda ^{1}b\left( \sigma ,n\right)
u^{\sigma }=0\text{ in }\Omega \text{ for each }\sigma .
\label{Macroscopic equation}
\end{equation}%
Choosing general $\psi ^{\sigma }\in H^{1}\left( \Omega \right) $ satisfying
(\ref{boundary condition on test function}) yields the boundary conditions%
\begin{equation}
\sum_{\sigma }c(\sigma ,n)u^{\sigma }\overline{\psi ^{\sigma }}=0\text{ on }%
\partial \Omega .  \label{boundary condition k diff 0}
\end{equation}%
We introduce the matrices $C_{1}=diag(\left( c(\sigma ,n)\right) _{\sigma })$%
, $C_{2}=diag(\left( b(\sigma ,n)\right) _{\sigma })$ and the vectors $%
u=(u^{\sigma })_{\sigma }$, $\psi =(\psi ^{\sigma })_{\sigma },$ $\varphi
=\left( \phi _{n}^{\sigma }\left( 0\right) e^{sign\left( \sigma \right)
2i\pi \frac{l^{k}x}{\alpha }}\right) _{\sigma }$ with $\sigma \in I^{k},$ so
that (\ref{boundary condition on test function}, \ref{Macroscopic equation}, %
\ref{boundary condition k diff 0}) can be written on the matrix form%
\begin{gather*}
C_{1}\partial _{x}u+\lambda ^{1}C_{2}u=0\text{ in }\Omega \text{ ,} \\
\text{and }C_{1}u(x).\overline{\psi }(x)=0\text{ on }\partial \Omega \text{
for all }\psi \text{ such that }\overline{\varphi }(x,0).\overline{\psi }%
(x)=0\text{ on }\partial \Omega .
\end{gather*}%
The boundary condition is equivalent to $C_{1}u(x)$ is collinear with $%
\overline{\varphi }(x,0)$ i.e. $\det (C_{1}u(x),\overline{\varphi }(x,0))=0$%
. Equivalently%
\begin{equation*}
\left\{
\begin{array}{l}
c(k,n)u^{k}\left( 0\right) \overline{\phi ^{-k}\left( 0\right) }%
-c(-k,n)u^{-k}\left( 0\right) \overline{{\phi }^{k}\left( 0\right) }=0, \\
c(k,n)u^{k}\left( \alpha \right) \overline{\phi ^{-k}\left( 0\right)
e^{-2i\pi l^{k}}}-c(-k,n)u^{-k}\left( \alpha \right) \overline{{\phi }%
^{k}\left( 0\right) e^{2i\pi l^{k}}}=0.%
\end{array}%
\right.
\end{equation*}%
Finally, since $c(k,n)=-c(-k,n)$ and $c(k,n)$ is assumed to do not vanish,
the boundary conditions of macroscopic equation (\ref{Macroscopic equation})
are%
\begin{equation}
u^{k}\left( x\right) \phi _{n}^{k}\left( 0\right) e^{2i\pi \frac{l^{k}x}{%
\alpha }}+u^{-k}\left( x\right) {\phi }_{n}^{-k}\left( 0\right) e^{-2i\pi
\frac{l^{k}x}{\alpha }}=0\text{ at }x\in \partial \Omega .  \notag
\end{equation}

\subsubsection{Case $k=0$\label{Case k=0}}

In case $k=0$, to avoid any confusion with $\lambda ^{0}$, the upper indices
$k=0$ are removed. We denote by $\phi _{n},\phi _{m}$ the eigenvectors
associated to $\lambda ^{0}=\lambda _{n}=\lambda _{m}$, solutions to $%
\mathcal{P}(0)$ in (\ref{Blochwaves}), and by $\sum_{p}$, $\sum_{q}$ the
sums over $p$ or $q$ varying in $\left\{ n,m\right\} $. We restart with a
test function
\begin{equation}
v^{\varepsilon }(x):=\mathfrak{R}(\sum_{p}v_{p})\in H_{0}^{1}\left( \Omega
\right) \cap H^{2}\left( \Omega \right)  \label{test_function_0}
\end{equation}%
for the very weak formulation (\ref{1D-weakcomplete}). We pose $%
v_{p}(x,y)=\psi _{p}(x)\phi _{p}(y)$ with $\psi _{p}(x)\in H^{1}\left(
\Omega \right) $ for $p\in \left\{ n,m\right\} .$ Since $\phi _{p}$ is
periodic thus $\phi _{p}(\frac{x}{\varepsilon })=\phi _{p}(0)$ at $x\in
\partial \Omega $ and the boundary condition in (\ref{test_function_0}) is
equivalent to
\begin{equation*}
\sum_{p}\psi _{p}(x)\phi _{p}(0)=0\text{ at }x\in \partial \Omega .
\end{equation*}%
By setting $c(p,q):=c(0,p,q)$ for $p,q\in \left\{ n,m\right\} $, using the
expression in Lemma \ref{Lemme-decomposition-Bloch} of the weak limit $w^{0}$
of $S_{0}^{\varepsilon }w^{\varepsilon }$,%
\begin{equation}
w^{0}\left( x,y\right) =\sum_{p}u_{p}\left( x\right) \phi _{p}(y),
\label{1D-decompose-0}
\end{equation}%
using the periodicity of $\left( \phi _{p}\right) _{p}$ and conducting the
same calculations as for $k\neq 0$, we obtain%
\begin{equation*}
\sum_{p,q}\left[ \int_{\Omega }(c(p,q)\partial _{x}u_{q}+\lambda
^{1}b(p,q)u_{q})\cdot \psi _{p}\text{ }dx-\int_{\partial \Omega
}c(p,q)u_{q}\cdot \psi _{p}\text{ }dx\right] =0.
\end{equation*}%
With $u=(u_{p})_{p}$, $\psi =(\psi _{p})_{p},~\phi =\left( \phi _{p}\right)
_{p}$ and $C_{1}=(c(p,q))_{p,q}$, $C_{2}=(b(p,q))_{p,q},$ the macroscopic
problem turns to be
\begin{equation}
C_{1}\partial _{x}u+\lambda ^{1}C_{2}u=0\text{ in }\Omega ,
\label{diff_equation_w}
\end{equation}%
with the boundary conditions
\begin{equation*}
C_{1}u(x).\psi (x)=0\text{ on }\partial \Omega \text{ for all }\psi \text{
such that }\psi (x).\phi (0)=0\text{ on }\partial \Omega \text{.}
\end{equation*}%
Equivalently, $C_{1}u(x)$ is collinear to $\phi (0)$ on $\partial \Omega $ or%
\begin{equation}
\det \left( C_{1}u(x),\phi (0)\right) =0\text{ on }\partial \Omega .
\label{boundary_k_0}
\end{equation}%
But $c(p,p)=0$, so (\ref{boundary_k_0}) simplifies to%
\begin{equation*}
\left\{
\begin{array}{c}
c\left( n,m\right) u_{m}\left( 0\right) \phi _{m}\left( 0\right) -c\left(
m,n\right) u_{n}\left( 0\right) \phi _{n}\left( 0\right) =0, \\
c\left( n,m\right) u_{m}\left( \alpha \right) \phi _{m}\left( 0\right)
-c\left( m,n\right) u_{n}\left( \alpha \right) \phi _{n}\left( 0\right) =0.%
\end{array}%
\right.
\end{equation*}%
Finally, since $c\left( n,m\right) =-c\left( m,n\right) $ and $c\left(
n,m\right) \neq 0$, the boundary conditions are%
\begin{equation*}
u_{n}\left( x\right) \phi _{n}\left( 0\right) +u_{m}\left( x\right) \phi
_{m}\left( 0\right) =0\text{ on }\partial \Omega .
\end{equation*}

\subsection{Analytic solutions\label{analytic_solution}}

For $k\in Y^{\ast }$ and $\rho =1$, we solve the macroscopic equations In
Section \ref{macro-solution}. These solutions are used to validate the
numerical results in the final Section. Moreover, in Section \ref%
{Case-a-rho-known}, the exact formulations of the two-scale eigenmodes are
found for $\rho =1$ and $a=1$.

\subsubsection{The case\textbf{\ }$\protect\rho =1$\label{macro-solution}}

For $k\neq 0$ and $b\left( n,n\right) =1$, the exact solutions of the
macroscopic equation (\ref{marco_1}) are
\begin{equation*}
u_{n}^{\sigma }\left( x\right) =d^{\sigma }e^{-\lambda ^{1}c\left( \sigma
,n\right) ^{-1}x}\text{ for each }\sigma \in I^{k}
\end{equation*}%
where $d^{\sigma }$ is any complex number. Applying the boundary condition (%
\ref{boundary_macro_1}) and assuming that $\phi _{n}^{k}\left( 0\right) \neq
0,$ the eigenvalue is%
\begin{equation}
\lambda ^{1}=\frac{c(k,n)}{\alpha }\left( 2i\pi l^{k}-i\ell \pi \right)
\text{ for }\ell \in \mathbb{Z}\text{.}  \label{lambda-1k}
\end{equation}%
Furthermore, $u_{n}^{k}=\overline{u_{n}^{-k}}$ and $\phi _{n}^{k}\left(
0\right) =\overline{\phi _{n}^{-k}}\left( 0\right) $ then $\func{Re}\left(
d^{k}\phi _{n}^{k}\left( 0\right) \right) =0$, or $d^{k}\phi _{n}^{k}\left(
0\right) =i\delta $ for any $\delta \in
\mathbb{R}
$. Thus,
\begin{equation*}
d^{k}=\frac{i\delta }{\phi _{n}^{k}\left( 0\right) }\text{ and }d^{-k}=-%
\frac{i\delta }{\phi _{n}^{-k}\left( 0\right) }\text{ for any }\delta \in
\mathbb{R}
\text{.}
\end{equation*}

For $k=0$, using the equalities $c\left( n,n\right) =c\left( m,m\right) =0,$
$b\left( n,m\right) =b\left( m,n\right) =0$ and $b\left( n,n\right) =b\left(
m,m\right) =1$, the macroscopic equation (\ref{marco_1_0}) is rewritten%
\begin{equation}
\left\{
\begin{array}{l}
c\left( n,m\right) {\partial _{x}u_{m}^{0}+\lambda }^{1}{u_{n}^{0}}=0\quad
\text{in}\quad \Omega , \\
c\left( m,n\right) {\partial _{x}u_{n}^{0}+\lambda }^{1}{u_{m}^{0}}=0\quad
\text{in}\quad \Omega .%
\end{array}%
\right.  \label{macro_periodic}
\end{equation}

If\textbf{\ }$\lambda ^{1}=0$, ${\partial _{x}u_{m}^{0}=0}$ and ${\partial
_{x}u_{n}^{0}=0}$ in $\Omega $, then ${u_{m}^{0}}$ and ${u_{n}^{0}}$ are
independent on $x$, equivalently, ${u_{m}^{0}}$ and ${u_{n}^{0}}$ are
complex numbers.

If ${\lambda }^{1}\neq 0$, the first equation gives $u_{n}^{0}=-\frac{%
c\left( n,m\right) {{\partial _{x}}u_{m}^{0}}}{{\lambda }^{1}}$ in $\Omega $
and since $c(n,m)=-c(m,n)$ then%
\begin{equation}
{\partial _{xx}}u_{m}^{0}=-\left( \frac{{{{{\lambda }^{1}}}}}{c\left(
n,m\right) }\right) ^{2}u_{m}^{0}  \label{periodic_spectral_0}
\end{equation}%
and%
\begin{equation*}
u_{m}^{0}\left( x\right) ={d_{1}}\cos \left( \frac{{\lambda }^{1}}{c\left(
n,m\right) }{x}\right) +{d_{2}}\sin \left( \frac{{\lambda }^{1}}{c\left(
n,m\right) }{x}\right)
\end{equation*}%
for two constants for $d_{1},d_{2}\in \mathbb{C}$ and $u_{n}^{0}$ follows by
its above expression. Applying the boundary condition (\ref%
{boundary_macro_1_0}), if $\phi _{m}^{0}\left( 0\right) \neq 0$,%
\begin{equation}
\lambda ^{1}=\frac{\ell \pi c\left( n,m\right) }{\alpha }\text{ for }\ell
\in \mathbb{Z\,}\text{\ and }{d_{1}=}-{{d_{2}}}\frac{\phi _{n}^{0}\left(
0\right) }{\phi _{m}^{0}\left( 0\right) }  \label{lambda-10}
\end{equation}%
for any $\ell \in \mathbb{Z}$ and ${d_{2}}\in \mathbb{C}$. If $\phi
_{m}^{0}\left( 0\right) =0$ then $\phi _{n}^{0}\left( 0\right) =0$ or $%
u_{n}^{0}\left( x\right) =0$ on $\partial \Omega .$ In the case $\phi
_{n}^{0}\left( 0\right) =0$, the macroscopic equation is lacking of boundary
conditions and their solutions are not unique, they depend on arbitrary
coefficients $d_{1},d_{2}$ and $\lambda ^{1}$. When $u_{n}^{0}\left(
x\right) =0$ at $\partial \Omega $, there is an alternative, or $u_{n}^{0}$
is the trivial solution or%
\begin{equation*}
\det \left(
\begin{array}{cc}
0 & 1 \\
-\sin \left( \frac{\lambda _{1}}{c\left( n,m\right) }\alpha \right) & \cos
\left( \frac{\lambda _{1}}{c\left( n,m\right) }\alpha \right)%
\end{array}%
\right) =0
\end{equation*}%
and then $d_{2}=0$, $\lambda ^{1}=\frac{\ell \pi c\left( n,m\right) }{\alpha
}$ for any $\ell \in \mathbb{Z}$ and $d_{1}\in
\mathbb{C}
$.

\bigskip

\begin{remark}
\label{interval for lambda^1}According to (\ref{lambda-1k}) and (\ref%
{lambda-10}), $\lambda ^{1}=0$ iff $\ell =2l^{k}$ for $k\neq 0$ and iff $%
\ell =0$ otherwise. So, in any case small values of ${\lambda }^{1,\ell }$
correspond to indices $\ell $ in a vicinity of $2l^{k}$ or to $\frac{%
2k\alpha }{\varepsilon }$ when $\varepsilon >0$.
\end{remark}

\subsubsection{The case\textbf{\ }$a=\protect\rho =1$\label{Case-a-rho-known}%
}

We consider the spectral problem%
\begin{equation*}
-\partial _{yy}^{2}\phi ^{k}=\lambda ^{k}\phi ^{k}\text{ in }Y
\end{equation*}%
with the $k-$quasi-periodicity conditions.

For $k\neq 0$, for a mapping $m\mapsto n(m)$ from $\mathbb{Z}$ to $\mathbb{N}%
^{\ast }$ not detailed here, $\lambda _{n(m)}^{k}=4\pi ^{2}(m+k)^{2}$ and
there are exactly two conjugated solutions $\phi _{n(m)}^{\sigma
}(y)=e^{sign\left( \sigma \right) 2i\pi (m+k)y}$ for any $m\in \mathbb{Z}$
and $\sigma \in I^{k}$. It follows that $c(\sigma ,n(m))=sign\left( \sigma
\right) 4i\pi (m+k)$, $b\left( \sigma ,n(m)\right) =1$ and $\lambda ^{1}=-%
\frac{4\pi ^{2}}{\alpha }(2l^{k}-\ell )(m+k)$ for any $\ell \in \mathbb{Z}$,
so%
\begin{equation*}
u_{n(m)}^{\sigma }(x)=d^{\sigma }e^{\frac{sign\left( \sigma \right) i\pi }{%
\alpha }(2l^{k}-\ell )x}
\end{equation*}%
and the resulting two-scale eigenmode is
\begin{equation*}
w^{\sigma }(x,y)=d^{\sigma }e^{\frac{sign\left( \sigma \right) i\pi }{\alpha
}(2l^{k}-\ell )x}e^{sign\left( \sigma \right) 2i\pi (n+k)y}.
\end{equation*}%
For $k=0$, for each $\lambda _{n(m)}^{0}=(2\pi m)^{2}$ there are two
eigenvectors $\phi _{n(m)}(y)=\cos (2\pi my)$ and $\phi _{n(m)+1}(y)=\sin
(2\pi my)$ so
\begin{equation*}
C_{1}=2m\pi \left(
\begin{array}{cc}
0 & 1 \\
-1 & 0%
\end{array}%
\right) ,\text{ }C_{2}=\frac{1}{2}\left(
\begin{array}{cc}
1 & 0 \\
0 & 1%
\end{array}%
\right) ,\text{ }\left(
\begin{array}{c}
\phi _{n(m)}(0) \\
\phi _{n(m)+1}(0)%
\end{array}%
\right) =\left(
\begin{array}{c}
1 \\
0%
\end{array}%
\right) \text{.}
\end{equation*}%
It implies that $\lambda ^{1}=\frac{4m\ell \pi ^{2}}{\alpha }$ for any $\ell
\in \mathbb{Z}$ and%
\begin{equation*}
u_{n(m)}(x)=d_{0}\sin \left( \ell \pi \frac{{x}}{\alpha }\right) \text{ and }%
u_{n(m)+1}(x)=d_{0}\cos (\ell \pi \frac{{x}}{\alpha }),
\end{equation*}%
then the two-scale eigenmode is%
\begin{equation*}
w(x,y)=d_{0}[\sin \left( \ell \pi \frac{{x}}{\alpha }\right) \cos (2\pi
my)+\cos (\ell \pi \frac{{x}}{\alpha })\sin (2\pi my)]\text{ for }\ell ,m\in
\mathbb{Z}\text{.}
\end{equation*}

\subsection{Neumann boundary conditions}

We consider the spectral problem with Neumann boundary conditions
\begin{equation*}
{P}^{\varepsilon }{w^{\varepsilon }}=\lambda ^{\varepsilon }\rho
^{\varepsilon }w^{\varepsilon }\quad \text{in}\quad \Omega ~\text{\ \ and \ }%
\partial _{x}w^{\varepsilon }=0\quad \text{on}\quad \partial \Omega .
\end{equation*}%
The process of homogenization and the results are similar to the case of
Dirichlet boundary conditions. The microscopic problem and the internal
macroscopic equation are unchanged while the boundary conditions of the
latter are%
\begin{equation*}
\sum_{\sigma \in I^{k}}\sum_{m}u_{m}^{\sigma }\left( x\right) \partial
_{y}\phi _{m}^{\sigma }\left( 0\right) e^{sign\left( \sigma \right) 2i\pi
\frac{l^{k}x}{\alpha }}=0\text{ on }\partial \Omega
\end{equation*}%
where the cases $k\neq 0$ and $k=0$ are not separated so a general notation
is adopted for the sum over $m$ and $\sigma $. Their derivation follows the
same steps, so we only mention the boundary condition satisfied by the test
functions. They are chosen to satisfy $\partial _{x}v^{\varepsilon }\left(
x\right) =0$ on $\partial \Omega $ or equivalently,%
\begin{equation*}
\sum_{\sigma \in I^{k}}\sum_{m}\partial _{x}\psi _{m}^{\sigma }\left(
x\right) \phi _{m}^{\sigma }\left( \frac{x}{\varepsilon }\right) +\frac{1}{%
\varepsilon }\psi _{m}^{\sigma }\left( x\right) \partial _{y}\phi
_{m}^{\sigma }\left( \frac{x}{\varepsilon }\right) =0\text{\ on }\partial
\Omega .
\end{equation*}%
Multiplying by $\varepsilon ,$
\begin{equation}
\sum_{\sigma \in I^{k}}\sum_{m}\psi _{m}^{\sigma }\left( x\right) \partial
_{y}\phi _{m}^{\sigma }\left( \frac{x}{\varepsilon }\right) +O(\varepsilon
)=0\text{ on }\partial \Omega ,  \label{boundary_cond_new}
\end{equation}%
then using the $\sigma -$quasi-periodicity of $\phi _{m}^{\sigma }$ and
passing to the limit%
\begin{equation*}
\sum_{\sigma \in I^{k}}\sum_{m}\psi _{m}^{\sigma }\left( x\right) \partial
_{y}\phi _{m}^{\sigma }\left( 0\right) e^{sign\left( \sigma \right) 2i\pi
\frac{l^{k}x}{\alpha }}=0\text{ on }\partial \Omega .
\end{equation*}

\section{Homogenization based on a first order formulation \label{first}}

In this section, the homogenized model is derived based on a first order
formulation. The calculations are less detailed than in Section \ref{second}%
, only the main results and the proof principles are given.

\subsection{Reformulation of the spectral problem and the main result}

We start by setting
\begin{equation*}
U^{\varepsilon }=\left( \frac{{\sqrt{a^{\varepsilon }}\partial _{x}}%
w^{\varepsilon }}{{i\sqrt{\lambda ^{\varepsilon }}}},\sqrt{\rho
^{\varepsilon }}w^{\varepsilon }\right) ,\text{ }\mu ^{\varepsilon }=\sqrt{%
\lambda ^{\varepsilon }},
\end{equation*}%
\begin{equation*}
A^{\varepsilon }=\left( {%
\begin{array}{cc}
0 & {\sqrt{a^{\varepsilon }}\partial _{x}\left( \frac{1}{\sqrt{\rho
^{\varepsilon }}}.\right) } \\
\frac{1}{\sqrt{\rho ^{\varepsilon }}}{\partial _{x}\left( {\sqrt{%
a^{\varepsilon }}.}\right) } & 0%
\end{array}%
}\right) ,\text{ }n_{A^{\varepsilon }}=\frac{1}{\sqrt{\rho ^{\varepsilon }}}%
\left(
\begin{array}{cc}
0 & {\sqrt{a^{\varepsilon }}n}_{\Omega } \\
{\sqrt{a^{\varepsilon }}n}_{\Omega } & 0%
\end{array}%
\right)
\end{equation*}%
with the domain of the operator $A^{\varepsilon }$,%
\begin{equation*}
D\left( {A^{\varepsilon }}\right) :=\left\{ {\left( {\varphi ,\phi }\right)
\in {L^{2}}\left( \Omega \right) \times {L^{2}}\left( \Omega \right)
\left\vert {\sqrt{{a^{\varepsilon }}}\varphi \in {H^{1}}\left( \Omega
\right) ,{\phi }\in {H_{0}^{1}}\left( \Omega \right) }\right. }\right\}
\subset L^{2}(\Omega )^{2},
\end{equation*}%
so that $iA^{\varepsilon }$ is self-adjoint on $L^{2}(\Omega )^{2}$ as
proved in \cite{brassart2010two}. The spectral equation (\ref{a2_1}) can be
recasted as a first-order system%
\begin{equation}
A^{\varepsilon }U^{\varepsilon }=i\mu ^{\varepsilon }U^{\varepsilon }\text{
~in }\Omega \text{ \ and }U_{2}^{\varepsilon }=0\text{ on }\partial \Omega ,
\label{a6_2}
\end{equation}%
where $U_{2}^{\varepsilon }$ is the second component of $U^{\varepsilon }$.
We observe that $\left\Vert \sqrt{\rho ^{\varepsilon }}w^{\varepsilon
}\right\Vert _{L^{2}\left( \Omega \right) }\leq \left\Vert \sqrt{\rho
^{\varepsilon }}\right\Vert _{L^{\infty }\left( \Omega \right) }$ and that $%
\left\Vert \frac{\sqrt{a^{\varepsilon }}\partial _{x}w^{\varepsilon }}{{i}%
\sqrt{\lambda ^{\varepsilon }}}\right\Vert _{L^{2}\left( \Omega \right)
}\leq M_{0}$ can be deduced from the weak formulation (\ref%
{1D-weakformulation1}), therefore $U^{\varepsilon }$ is uniformly bounded,%
\begin{equation}
\left\Vert U^{\varepsilon }\right\Vert _{L^{2}\left( \Omega \right)
}^{2}\leq M_{1}.  \label{1D-bounded-U}
\end{equation}%
We start our analysis from the system expressed in a distributional sense,
\begin{equation}
\int_{\Omega }{U^{\varepsilon }\cdot \left( {i\mu ^{\varepsilon
}-A^{\varepsilon }}\right) \Psi ~dx}=0,  \label{a6_4}
\end{equation}%
for all admissible test functions $\Psi $ $=\left( \varphi ,\psi \right) {%
\in H^{1}\left( \Omega \right) \times H_{0}^{1}\left( \Omega \right) }$. We
choose $\mu _{0}=\sqrt{\lambda ^{0}}$ and $\mu _{1}=\frac{\lambda ^{1}}{2\mu
_{0}}$, so $\mu ^{\varepsilon }$ can be decomposed as
\begin{equation}
\mu ^{\varepsilon }=\frac{\mu _{0}}{\varepsilon }+\mu _{1}+O\left(
\varepsilon \right) .  \label{a7_3}
\end{equation}%
The asymptotic spectral problem (\ref{Blochwaves}) is also restated as a
first order system by setting%
\begin{equation*}
A_{k}:=\left( {%
\begin{array}{cc}
0 & {\sqrt{a}\partial _{y}\left( \frac{1}{\sqrt{\rho }}.\right) } \\
\frac{1}{\sqrt{\rho }}{\partial _{y}\left( {\sqrt{a}.}\right) } & 0%
\end{array}%
}\right) \text{ and }n_{A_{k}}=\frac{1}{\sqrt{\rho }}\left(
\begin{array}{cc}
0 & {\sqrt{a}n}_{Y} \\
{\sqrt{a}n}_{Y} & 0%
\end{array}%
\right) ,
\end{equation*}%
and
\begin{equation}
e_{n}^{k}:=\frac{1}{\sqrt{2}}\left(
\begin{array}{c}
-i\frac{{s_{n}}}{{\ \sqrt{\lambda _{\left\vert n\right\vert }^{k}}}}\sqrt{a}%
\partial _{y}\left( {\phi _{\left\vert n\right\vert }^{k}}\right) \\
\sqrt{\rho }\phi _{\left\vert n\right\vert }^{k}%
\end{array}%
\right) ~\text{\ and }\mu _{n}^{k}=s_{n}\sqrt{\lambda _{\left\vert
n\right\vert }^{k}}\text{ for all }n\in
\mathbb{Z}
^{\ast },  \label{a3_2}
\end{equation}%
$s_{n}$ denoting the sign of $n.$ As proved in \cite{brassart2010two}, $%
iA_{k}$ is self-adjoint on the domain%
\begin{equation*}
D\left( {A_{k}}\right) :=\left\{ {\left( {\varphi ,\phi }\right) \in
L^{2}\left( Y\right) ^{2}|\sqrt{a}\varphi \in H_{k}^{1}\left( Y\right) ,}%
\frac{{\phi }}{\sqrt{\rho }}{\in H_{k}^{1}\left( Y\right) }\right\} \subset
L^{2}\left( Y\right) ^{2}.
\end{equation*}%
The Bloch wave spectral problem $\mathcal{P(}k\mathcal{)}$ is equivalent to
finding pairs $\left( \mu _{n}^{k},e_{n}^{k}\right) $ indexed by\thinspace $%
n\in
\mathbb{Z}
^{\ast }$ solution to%
\begin{equation}
\mathcal{Q(}k\mathcal{)}:A_{k}e_{n}^{k}=i\mu _{n}^{k}e_{n}^{k}\text{ \ in }Y%
\text{ with }e_{n}^{k}\in H_{k}^{1}\left( Y\right) ^{2}.
\label{Bloch wave first order}
\end{equation}%
The corresponding weak formulation is
\begin{equation}
\int_{Y}e_{n}^{k}\cdot \left( A_{k}-i\mu _{n}^{k}\right) \Psi \,dy=0\text{
for all }\Psi \in D\left( A_{k}\right) .  \label{Bloch-weak}
\end{equation}%
The relation between the operator $A^{\varepsilon }$ and the scaled operator
$A_{k}$ is obtained by considering any regular vector $\psi =\psi \left(
x,y\right) $ depending on both space scales,
\begin{equation}
A^{\varepsilon }\left( {\psi \left( {x,\frac{x}{\varepsilon }}\right) }%
\right) =\left( {\left( {\frac{1}{\varepsilon }A_{k}+B}\right) \psi }\right)
\left( {x,\frac{x}{\varepsilon }}\right) ,  \label{a3_3}
\end{equation}%
where the operator $B$ is defined as the result of the formal substitution
of $x-$derivatives by $y-$derivatives in $A_{k}$, i.e.
\begin{equation*}
B:=\left( {%
\begin{array}{cc}
0 & {\sqrt{a}\partial _{x}\left( \frac{1}{\sqrt{\rho }}.\right) } \\
\frac{1}{\sqrt{\rho }}{\partial _{x}\left( {\sqrt{a}.}\right) } & 0%
\end{array}%
}\right) .
\end{equation*}%
For any $n\in {%
\mathbb{Z}
}^{\ast }$ and $k\in Y^{\ast }$, $M_{n}^{k}:=\left\{ {i\in {%
\mathbb{Z}
}}^{\ast }\text{ }{|}\text{ }{\mu _{i}^{k}=\mu _{n}^{k}}\right\} $ is the
set of indices of eigenvectors related to the same eigenvalue ${\mu _{n}^{k}}
$.\ For all $k\in Y^{\ast }\diagdown \left\{ 0\right\} ,$ since $\mu
_{n}^{k}=\mu _{n}^{-k}$ then $M_{n}^{k}=M_{n}^{-k}$.

\begin{remark}
From now on, we shall assume that the weak limit of $S_{k}^{\varepsilon
}U^{\varepsilon }$\ in $L^{2}\left( \Omega \times Y\right) $ is not
vanishing to avoid eigenmodes related to the boundary spectrum (see
Proposition 7.7 in \cite{allaire1998bloch}).
\end{remark}

\begin{theorem}
\label{a7_7bis}For $k\in Y^{\ast }$, let $\left( \mu ^{\varepsilon
},U^{\varepsilon }\right) $ be solution of (\ref{a6_2}) then $\sum_{\sigma
\in I^{k}}S_{\sigma }^{\varepsilon }U^{\varepsilon }$ is bounded in $%
L^{2}\left( \Omega \times Y\right) $. For $\varepsilon \in E_{k}$, assuming
that the renormalized sequence $\varepsilon \mu ^{\varepsilon }$ satisfies
the decomposition (\ref{a7_3}) with $\mu _{0}=\mu _{n}^{k}$ an eigenvalue of
the Bloch wave spectrum, any weak limit $G_{k}$ of $\sum_{\sigma \in
I^{k}}S_{\sigma }^{\varepsilon }U^{\varepsilon }$ in $L^{2}\left( \Omega
\times Y\right) $ has the form%
\begin{equation}
G_{k}\left( {x,y}\right) =\sum_{\sigma \in I^{k}}\sum\limits_{m\in
M_{n}^{\sigma }}{u_{m}^{\sigma }\left( x\right) e_{m}^{\sigma }\left(
y\right) },  \label{a7_10}
\end{equation}%
where $\left( u_{m}^{\sigma }\right) _{m,\sigma }$ are the solutions of the
macroscopic equations (\ref{marco_1}, \ref{boundary_macro_1}) or (\ref%
{marco_1_0}, \ref{boundary_macro_1_0}).
\end{theorem}

Therefore, the physical solution $U^{\varepsilon }$ can be approximated by
\begin{equation}
U^{\varepsilon }\left( x\right) \approx \sum_{\sigma \in
I^{k}}\sum\limits_{m\in M_{n}^{\sigma }}{u_{m}^{\sigma }\left( x\right)
e_{m}^{\sigma }\left( {\frac{x}{\varepsilon }}\right) }.
\label{Physic_approximation1}
\end{equation}

\begin{proof}
For a given $k\in Y^{\ast }$, let $U^{\varepsilon }$ be solution of \ (\ref%
{a6_2}) which is bounded in $L^{2}(\Omega )$, the property (\ref{Pseudo_sk})
yields the boundness of $\left\Vert S_{\sigma }^{\varepsilon }U^{\varepsilon
}\right\Vert _{L^{2}\left( \Omega \times Y\right) }$. So there exist $%
U^{\sigma }\in $ $L^{2}(\Omega \times Y)^{2}$ such that, up the extraction
of a subsequence, $S_{\sigma }^{\varepsilon }U^{\varepsilon }$ tends weakly
to $U^{\sigma }$ in $L^{2}(\Omega \times Y)^{2}$ and hence, $%
\sum\limits_{\sigma \in I^{k}}S_{\sigma }^{\varepsilon }U^{\varepsilon }$
converges to $G_{k}\left( x,y\right) =\sum\limits_{\sigma \in
I^{k}}U^{\sigma }\left( x,y\right) $. Using the decomposition (\ref%
{decompose_bloch}) of $U^{\sigma }$ in the forthcoming Lemma \ref%
{Bloch_first},%
\begin{equation*}
G_{k}\left( x,y\right) =\sum\limits_{\sigma \in I^{k}}\sum\limits_{m\in
M_{n}^{\sigma }}u_{m}^{\sigma }\left( x\right) e_{m}^{\sigma }\left( y\right)
\end{equation*}%
The macroscopic problem solved by the coefficients $\left( u_{m}^{\sigma
}\right) _{\sigma ,m}$ is derived in Section \ref{Bloch2}.
\end{proof}

\subsection{Model derivation}

\subsubsection{Modal decomposition on the Bloch modes \label{Bloch1}}

\begin{lemma}
\label{Bloch_first}Let a sequence $\left( \mu ^{\varepsilon },U^{\varepsilon
}\right) $ be solution of (\ref{a6_2}) and satisfies (\ref{a7_3}) with $\mu
_{0}=\mu _{n}^{k}$ for given $n\in \mathbb{Z}^{\ast }$ and $k\in Y^{\ast }$,
we extract a subsequence of $\varepsilon $, still denoted by $\varepsilon $,
such that $S_{k}^{\varepsilon }U^{\varepsilon }$ converges weakly to $U^{k}$
in $L^{2}\left( \Omega \times Y\right) ^{2}$. If $U^{k}\in D\left(
A_{k}\right) $ then $\left( \mu _{n}^{k},U^{k}\right) $ is solution of the
Bloch wave equation (\ref{Bloch wave first order}) and $U^{k}$ admits the
modal decomposition
\begin{equation}
U^{k}\left( x,y\right) =\sum_{m\in M_{n}^{k}}{u}_{m}^{k}\left( x\right)
e_{m}^{k}\left( y\right) \text{ with }u_{m}^{k}\in L^{2}\left( \Omega
\right) .  \label{decompose_bloch}
\end{equation}
\end{lemma}

\begin{proof}
For each $k\in Y^{\ast }$, taking $\Psi \left( x,y\right) :=\theta (x)\phi
(y)$ with $\theta (x)\in C_{c}^{\infty }\left( \Omega \right) $ and $\phi
(y)\in {C^{\infty }}{\left( Y\right) ^{2}}$ $k-$quasi-periodic in $y$,
considering $\Re \Psi $ as a test functions in (\ref{a6_4}), and using (\ref%
{a3_3},\ref{a7_3}),
\begin{equation*}
\int_{\Omega }{U^{\varepsilon }\cdot \Re \left( {i\frac{{\mu _{0}}}{%
\varepsilon }+i\mu _{1}-\frac{{A_{k}}}{\varepsilon }-B}\right) \Psi
~dx+O\left( \varepsilon \right) =0}\text{.}
\end{equation*}%
Multiplying by $\varepsilon $%
\begin{equation*}
\int_{\Omega }{U^{\varepsilon }\cdot \Re \left( {i{\mu _{0}}-{A_{k}}}\right)
\Psi ~dx+O\left( \varepsilon \right) =0},
\end{equation*}%
and passing to the limit thanks to Corollary \ref{two-scale-converge},
\begin{equation*}
\frac{1}{\left\vert Y\right\vert }\int_{\Omega \times Y}{U^{k}\cdot \left( {%
i\mu _{0}-A_{k}}\right) \Psi ~dxdy=0}
\end{equation*}%
which is the weak formulation of the Bloch wave equations. If in addition ${%
U^{k}\in D}\left( A_{k}\right) ,$ integrating by parts yields
\begin{equation}
\frac{1}{\left\vert Y\right\vert }\int_{\Omega \times Y}{\left( {A_{k}-i\mu
_{0}}\right) U^{k}\cdot \Psi ~dxdy}-\frac{1}{\left\vert Y\right\vert }%
\int_{\Omega \times \partial Y}{U^{k}\cdot n_{A_{k}}\Psi ~dxdy=0}
\label{a8_3}
\end{equation}%
\ providing in turn the strong formulation,%
\begin{equation}
A_{k}U^{k}=i{\mu _{0}}U^{k}\quad \text{in}\quad \Omega \times Y.
\label{a8_2}
\end{equation}%
Since the product of a periodic function by a $k-$quasi-periodic function is
$k-$quasi-periodic then $n_{A_{k}}\Psi $ is $k-$quasi-periodic in $y$.
Therefore, $U^{k}$ is $k-$quasi-periodic$\,\ $in$\,\ y$ and finally is a
Bloch eigenvector in $y$. By projection, it can be decomposed as
\begin{equation*}
U^{k}\left( x,y\right) =\sum_{m\in M_{n}^{k}}{u}_{m}^{k}\left( x\right)
e_{m}^{k}\left( y\right) \text{ with }u_{m}^{k}=\frac{1}{b\left(
k,m,m\right) }\int_{Y}U^{k}\cdot e_{m}^{k}\text{ }dy\in L^{2}\left( \Omega
\right) \text{.}
\end{equation*}
\end{proof}

\subsubsection{Derivation of the macroscopic equation \label{Bloch2}}

The macroscopic equation is stated for each $k\in Y^{\ast }$ and each
eigenvalue $\mu _{n}^{k}$ of the Bloch wave spectral problem $\mathcal{Q(}k%
\mathcal{)}$. We pose
\begin{equation}
\kappa \left( k,n,m\right) =\frac{-ic\left( k,n,m\right) }{2\mu _{0}}\text{
for }m\in M_{n}^{k}  \label{ka}
\end{equation}%
where $c\left( k,n,m\right) $ is defined in (\ref{def of b and c}) and
notice that
\begin{gather*}
\kappa \left( k,n,m\right) =-\kappa \left( -k,m,n\right) ,\text{ }\kappa
\left( k,n,m\right) =-\overline{\kappa \left( -k,n,m\right) },\text{ } \\
\kappa \left( k,n,m\right) =-\overline{\kappa \left( k,m,n\right) },\text{
and }\kappa \left( 0,n,n\right) =0\text{.}
\end{gather*}%
For the sake of simplicity, we do the proof for $n\in
\mathbb{Z}
^{\ast +}$ only and denote by $\kappa \left( k,n\right) =\kappa \left(
k,n,n\right) $ and $\kappa \left( n,m\right) =\kappa \left( 0,n,m\right) $.
For general $n$, the proof is the same but $\phi _{n}^{k}$ is replaced by $%
\phi _{\left\vert n\right\vert }^{k}$.

\paragraph{Case $k\neq 0$}

The pairs $\left( \mu _{n}^{k},e_{n}^{k}\right) $ and $\left( \mu
_{n}^{-k},e_{n}^{-k}\right) $ are the eigenmodes of the spectral equations $%
\mathcal{Q(\pm }k\mathcal{)}$ in (\ref{Bloch wave first order})
corresponding to the eigenvalue $\mu _{0}=\mu _{n}^{k}=\mu _{n}^{-k}$. We
pose $\Psi ^{\varepsilon }=\Re \left( \Psi ^{k}+\Psi ^{-k}\right) \in $ $%
H^{1}(\Omega )\times H_{0}^{1}(\Omega )$ as a test function in the weak
formulation (\ref{a6_4}), with each $\Psi ^{\sigma }\left( x,y\right) =\psi
^{\sigma }\left( x\right) e_{n}^{\sigma }(y)$ where $\psi ^{\sigma }\in
H^{1}(\Omega )$ and satisfies the boundary conditions,%
\begin{equation*}
\sum\limits_{\sigma }\psi ^{\sigma }\left( x\right) \phi _{n}^{\sigma
}\left( \frac{x}{\varepsilon }\right) =0\text{ on }\partial \Omega .
\end{equation*}%
Notice that this condition is related to the second component of $\Psi
^{\varepsilon }$ only. Proceeding as in Section \ref{k_non_0_single_value}
yields (\ref{boundary condition on test function}). Since $\left( i\mu
_{0}-A_{\sigma }\right) \Psi ^{\varepsilon }=0$ for all $\sigma ,$ applying (%
\ref{a7_3}, \ref{a3_3}), then Equation (\ref{a6_4}) yields
\begin{equation}
\sum_{\sigma }\int_{\Omega }{U^{\varepsilon }\cdot \Re \left( {i\mu _{1}-B}%
\right) \Psi }^{\sigma }~{dx+}O\left( \varepsilon \right) =0\text{.}
\label{a8_7}
\end{equation}%
But ${\left( {i{\mu _{1}}-B}\right) {\Psi }^{\sigma }}$ is $\sigma -$%
quasi-periodic so passing to the limit thanks to Corollary \ref%
{two-scale-converge},%
\begin{equation}
\frac{1}{{\left\vert Y\right\vert }}\sum_{\sigma }\int_{\Omega \times Y}{{%
U^{\sigma }}\cdot \left( {i{\mu _{1}}-B}\right) {{\Psi }^{\sigma }}dxdy}=0.
\label{a8_9}
\end{equation}%
From Lemma \ref{Bloch_first}, $U^{\sigma }$ is decomposed as
\begin{equation*}
U^{\sigma }\left( x,y\right) ={u}_{n}^{\sigma }\left( x\right) e_{n}^{\sigma
}\left( y\right) .
\end{equation*}%
After replacement,

\begin{equation*}
\sum_{\sigma }\int_{\Omega }\left( -{i{\mu _{1}b}}\left( \sigma ,n\right) {u}%
_{n}^{\sigma }\cdot {\psi }^{\sigma }{{+\kappa \left( \sigma ,n\right)
u_{n}^{\sigma }\cdot \partial }}_{x}{\psi }^{\sigma }\right) {dx}=0
\end{equation*}%
for all ${\psi }^{\sigma }\in H^{1}\left( \Omega \right) $ fulfilling (\ref%
{boundary condition on test function}). Moreover, if ${{u_{n}^{\sigma }\in H}%
}^{1}\left( \Omega \right) $ it satisfies the strong form of the internal
equations%
\begin{equation}
{\kappa }\left( \sigma ,n\right) {{{\partial }_{x}u_{n}^{\sigma }}}-{i{\mu
_{1}}b}\left( \sigma ,n\right) {u}_{n}^{\sigma }=0\,\ \text{in }\Omega \,%
\text{\ for all }\sigma \in I^{k},  \label{macroscopic}
\end{equation}%
and the boundary conditions
\begin{equation*}
\sum_{\sigma }{\kappa }\left( \sigma ,n\right) {{u_{n}^{\sigma }\cdot }\psi }%
^{\sigma }=0\text{ on }\partial \Omega \text{.}
\end{equation*}%
Following the same calculations as in Section \ref{k_non_0_single_value},
with the matrices $C_{1}=diag\left( {\kappa }\left( \sigma ,n\right) \right)
$, $C_{2}=diag\left( {b}\left( \sigma ,n\right) \right) $ and the vectors $%
u=\left( {{{u}_{n}^{\sigma }}}\right) _{\sigma },\psi =\left( \psi ^{\sigma
}\right) _{\sigma },\varphi =\left( \phi ^{\sigma }\left( 0\right)
e^{sign\left( \sigma \right) 2i\pi x\frac{l^{k}}{\alpha }}\right) _{\sigma }$%
, (\ref{macroscopic}) is written on the matrix form%
\begin{equation*}
C_{1}{{{\partial }_{x}{u}}}={i{\mu _{1}C}}_{2}{u}\,\ \text{in }\Omega \,%
\text{,}
\end{equation*}%
with boundary condition
\begin{equation*}
C_{1}{u}\left( x\right) {.}\overline{\psi }\left( x\right) =0\,\text{\ on }%
\partial \Omega \text{ for all }\psi \text{ such that }\overline{\varphi }%
\left( x,0\right) .\overline{\psi }\left( x\right) =0\,\text{\ on }\partial
\Omega \text{.}
\end{equation*}%
Equivalently, $C{u}\left( x\right) $ is collinear with $\overline{\varphi }%
\left( x,0\right) $ yielding the boundary conditions
\begin{equation}
u_{n}^{k}\left( x\right) \phi _{n}^{k}\left( 0\right) e^{2i\pi \frac{l^{k}x}{%
\alpha }}+u_{n}^{-k}\left( x\right) \phi _{n}^{-k}\left( 0\right) e^{-2i\pi
\frac{l^{k}x}{\alpha }}=0\text{ on }\partial \Omega  \label{macro_k_1}
\end{equation}%
after remarking that ${\kappa }\left( \sigma ,n\right) \neq 0$. Finally,
with (\ref{ka}) and $\lambda ^{1}=2\mu _{0}\mu _{1}$ the macroscopic problem
(\ref{marco_1}, \ref{boundary_macro_1}) is recovered.

\paragraph{ Case $k=0$ \label{k=0_first_order}}

We adopt the same simplifications of notations that in Section \ref{Case k=0}%
. Let $e_{n}$ and $e_{m}$ be the Bloch eigenmodes of $\mathcal{Q(}0\mathcal{)%
}$ in (\ref{Bloch wave first order}) regarding the double eigenvalue\textbf{%
\ }$\mu _{0}=\mu _{n}=\mu _{m}$. In this case $M_{n}^{0}=\left\{ n,m\right\}
$. Taking $\Psi ^{\varepsilon }=\sum\limits_{p\in M_{n}^{0}}\Re \left( \Psi
_{p}\right) \in H^{1}\left( \Omega \right) \times H_{0}^{1}\left( \Omega
\right) $ as a test function with $\Psi _{p}\left( x,y\right) =\psi
_{p}\left( x\right) e_{p}(y)$ and $\psi _{p}\in H^{1}(\Omega )$. Due to the
periodicity of $\phi _{p},$ the second component of $\Psi ^{\varepsilon }$
satisfies the boundary conditions
\begin{equation}
\sum_{p\in M_{n}^{0}}\psi _{p}\left( x\right) \phi _{p}\left( 0\right) =0%
\text{ on }\partial \Omega .  \label{boundary_first_0}
\end{equation}%
Following similar calculations as for the case $k\neq 0$, the weak limit ${U}%
^{0}$ of ${S_{0}^{\varepsilon }{U^{\varepsilon }}}$ in $L^{2}(\Omega \times
Y)^{2}$ is%
\begin{equation*}
U^{0}\left( x,y\right) =\sum_{p\in M_{n}^{0}}{u}_{p}\left( x\right)
e_{p}\left( y\right)
\end{equation*}%
and $u_{p}$ is solution to the weak formulation%
\begin{equation*}
{{\sum_{q\in M_{n}^{0}}}}\int_{\Omega }-{i{\mu _{1}b}}\left( p,q\right) {u}%
_{q}\cdot {\psi }_{p}+{{\ {\kappa }\left( p,q\right) u_{q}\cdot {\partial }%
_{x}}}\psi _{p}~{dx}=0
\end{equation*}
for all ${\psi }_{p}\in H^{1}\left( \Omega \right) ~$with $p\in M_{n}^{0}$.
If $u_{q}\in H^{1}\left( \Omega \right) $ it is a solution to the internal
equations%
\begin{equation}
\sum_{q\in M_{n}^{0}}{{\ {\kappa }\left( p,q\right) {\partial }_{x}u_{q}}}-{i%
{\mu _{1}}\ {b}\left( p,q\right) u}_{q}=0\,\ \text{in }\Omega \,\text{ for }%
p\in M_{n}^{0},  \label{periodic_macro_equation_1}
\end{equation}%
and to the boundary conditions%
\begin{equation*}
\int_{\partial \Omega }\sum_{p,q\in M_{n}^{0}}{{\ {\kappa }\left( p,q\right)
u}}_{q}{{\cdot }\psi }_{p}~dx=0.
\end{equation*}%
Here, with $C_{1}=\left( {{\ {\kappa }\left( p,q\right) }}\right) _{p,q}$, $%
C_{2}=\left( {\ b\left( p,q\right) }\right) _{p,q}$, $u=\left( {{{u}_{p}}}%
\right) _{p},~\psi =\left( {\psi }_{p}\right) _{p},~\phi =\left( \phi
_{p}\right) _{p}$,%
\begin{equation*}
C_{1}{{{\partial }_{x}{u}}}={i{\mu _{1}C}}_{2}{u}\,\ \text{in }\Omega \,%
\text{,}
\end{equation*}%
\begin{equation*}
\text{and }C{u}\left( x\right) {.}\overline{\psi }\left( x\right) =0\,\
\text{on }\partial \Omega \text{ for all }\psi \text{ such that }\phi \left(
0\right) .\overline{\psi }\left( x\right) =0\text{ on }\partial \Omega .
\end{equation*}%
But {{${\kappa }\left( p,p\right) $}}${=0}$, therefore
\begin{equation}
u_{n}\left( x\right) \phi _{n}\left( 0\right) +u_{m}\left( x\right) \phi
_{m}\left( 0\right) =0\text{ on }\partial \Omega .
\label{periodic_boundary_!}
\end{equation}%
As for $k\neq 0$, these macroscopic equations are equivalent to (\ref%
{marco_1_0}, \ref{boundary_macro_1_0}).

\section{Numerical simulations\label{numerical_simulation}}

We report simulations regarding comparisons of physical eigenmodes and their
approximation by two-scale modes for $\rho =1$. In Subsection \ref{problem1}%
, for each given high frequency physical eigenelement a two-scale
eigenelement realizing a good approximation is identified. This shows that
the two-scale model can actually be used as an approximation of the complete
high-frequency spectra. Conversely, Subsection \ref{problem2} addresses the
modeling problem i.e. it introduces a way to generate approximations of
high-frequency spectra from the two-scale model only. Finally, in \ref%
{problem3} the order of convergence with respect to $\varepsilon $ is
analyzed. The next section describes the main simulation parameters.

\subsection{Simulation methods and conditions}

Both, the physical spectral problem and the Bloch wave spectral problem are
discretized by a quadratic finite element method. The number of elements are
respectively denoted $N_{phys}$ and $N_{bloch}$. The implementation of the $%
k-$quasi-periodic boundary condition is achieved by elimination of the last
degree of freedom. More precisely, for $n\in \left\{
1,...,2N_{bloch}+1\right\} $ the node indices, $\phi _{n}$ a degree of
freedom of $\phi \ $a Bloch eigenmode and $\varphi _{n}$ the corresponding
quadratic Lagrange interpolation function,
\begin{equation*}
\phi \left( y\right) \simeq \sum\limits_{n=2}^{2N_{bloch}}\phi _{n}\varphi
_{n}+\phi _{1}\varphi _{1}+\phi _{2N_{bloch}+1}\varphi _{2N_{bloch}+1}\text{.%
}
\end{equation*}%
Using the relation $\phi \left( 1\right) =e^{2i\pi k}\phi \left( 0\right) $
and taking $\varphi _{1}+e^{2i\pi k}\varphi _{2N_{bloch}+1}$ as the first
base function allows to eliminate $\phi _{2N_{bloch}+1}$,%
\begin{equation*}
\phi \left( y\right) \simeq \sum\limits_{n=2}^{2N_{bloch}}\phi _{n}\varphi
_{n}+\phi _{1}\left( \varphi _{1}+e^{2i\pi k}\varphi _{2N_{bloch}+1}\right)
\text{.}
\end{equation*}%
The sets\ of indices considered in the simulations of high frequency
physical modes and Bloch modes are denoted by $\mathcal{J}^{\varepsilon }$
and $J^{k}$, the former being generally included in $(\alpha /2\varepsilon
,N_{phys}/2)$. The Bloch modes are calculated for $k\geq 0$ only, and the
other cases can be deduced by conjugation. For each Bloch eigenmode $\left(
\lambda _{n}^{k},\phi _{n}^{k}\right) $, the macroscopic solutions $\left(
\lambda ^{1,\ell },u_{m,\ell }^{k}\right) _{m,\ell }$ are given in Section %
\ref{macro-solution} with $\delta =1$ and $d_{2}=\phi _{m}^{0}\left(
0\right) $ for any $m$ such that $\lambda _{m}^{k}=\lambda _{n}^{k}$ and $%
\ell \in
\mathbb{Z}
$. In fact, according to Remark \ref{interval for lambda^1} the index $\ell $
should vary in $J_{n}^{k}=\left[ \frac{2k}{\varepsilon }\right] +\left\{
-r,...,r\right\} ,$ for a small integer $r$, so that only the first
macroscopic eigenmodes be taken into account. In the next discussions, we
use the following notations for the two-scale approximations of the
eigenvalues and eigenmodes exhibiting clearly their parameters $\varepsilon
,k,n$ and $\ell $,%
\begin{equation}
\gamma _{n,\ell }^{\varepsilon ,k}:=\lambda _{n}^{k}+\varepsilon \lambda
^{1,\ell }\text{\ and }\psi _{n,\ell }^{\varepsilon ,k}\left( x\right)
:=\sum_{\sigma \in I^{k}}\sum\limits_{m}{u_{m,\ell }^{\sigma }\left(
x\right) \phi _{m}^{\sigma }\left( {\frac{x}{\varepsilon }}\right) }\text{
for\ }\ell \in J_{n}^{k}\text{, }n\in J^{k}.  \label{num_3}
\end{equation}%
In the simulations reported in Sections \ref{problem1} and \ref{problem2}
only one physical problem is used, namely $\Omega =\left( 0,1\right) $, $%
a^{\varepsilon }\left( x\right) =\sin \left( 2\pi x/\varepsilon \right) +2$,
50 cells (i.e. $\varepsilon =1/50$), and $N_{phys}=2,000$. Other number of
cells are used in Section \ref{problem3} for the convergence analysis.
Consequently, the coefficient of the Bloch wave spectral problem is $a\left(
y\right) =\sin \left( 2\pi y\right) +2$. The set $Y^{\ast }$ of positive
wave numbers in $Y^{\ast }$ is discretized by $L_{125}^{\ast +}=\left\{
0,...,62/125\right\} $ with step $\Delta _{k}=1/125$ and $N_{bloch}=50$. The
subset of macroscopic eigenvalues is restricted by $r=15$.

The first ten graphs $(k\mapsto \lambda _{n}^{k})_{n=1,...,10}$ of Bloch
eigenvalues are described in Figure \ref{eigenvalue}. The graphs\ are
symmetric about the axis $k=0$ which confirms that $\lambda _{n}^{k}=\lambda
_{n}^{-k}$ as remarked in Notation \ref{Conjugate of a Bloch mode}.
Moreover, all eigenvalues $\lambda _{n}^{k}$ are simple for $k\neq 0$ and
double for $k\in \left\{ 0,\pm \frac{1}{2}\right\} $.
\begin{figure}[h]
\centering
\includegraphics[width=7cm]{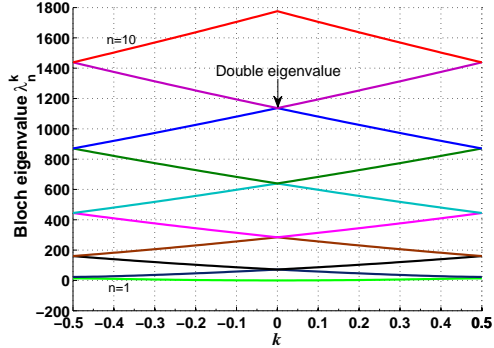}
\caption{First ten eigenvalues of the Bloch wave spectral problem.}
\label{eigenvalue}
\end{figure}

\subsection{Approximation of physical modes by two-scale modes \label%
{problem1}}

We discuss the approximation of a given solution $\left( \lambda
_{p}^{\varepsilon },w_{p}^{\varepsilon }\right) $ of Equation (\ref{a2_1})
for a given value of $\varepsilon $. From Remark \ref{exact_solution} we
expect to show numerically that there exists a suitable pair $(k,n)$ such
that the equality $\left( \lambda _{p}^{\varepsilon },w_{p}^{\varepsilon
}\right) =(\gamma _{n,\ell }^{\varepsilon ,k},\psi _{n,\ell }^{\varepsilon
,k})$ is exact with $(\gamma _{n,\ell }^{\varepsilon ,k},\psi _{n,\ell
}^{\varepsilon ,k})$ defined in (\ref{num_3}) and $\lambda ^{1,\ell }=0$.
Moreover, in the perspective of Remark \ref{Application to the wave equation}%
, $k$ varies in $L_{125}^{\ast +}$ only and approximations with $\lambda
^{1,\ell }\neq 0$ are expected. Whatever if $\lambda ^{1,\ell }$ vanishes or
not, we expect to search approximations for both eigenvalues and
eigenvectors which turns to be an multi-objective optimization problem that
might be solved by a dedicated method. However, to reduce the computational
cost, we propose an alternate approach consisting in minimizing the error on
eigenvalues in the approximation (\ref{eigenvalue_decomposition}),%
\begin{equation}
er_{value}\left( k\right) =\min_{n\in \mathbb{N},\text{ }\ell \in
J_{n}^{k}}\left\vert \frac{\varepsilon ^{2}\lambda _{p}^{\varepsilon
}-\gamma _{n,\ell }^{\varepsilon ,k}}{\varepsilon ^{2}\lambda
_{p}^{\varepsilon }}\right\vert ,  \label{num4}
\end{equation}%
for each $k\in L_{125}^{\ast +}$, and then in finding which one minimizes%
\begin{equation*}
er_{vector}\left( k\right) =\frac{\left\Vert w_{p}^{\varepsilon }-\psi
_{n_{k},\ell _{k}}^{\varepsilon ,k}\right\Vert _{L^{2}\left( \Omega \right) }%
}{\left\Vert w_{p}^{\varepsilon }\right\Vert _{L^{\infty }(\Omega )}}
\end{equation*}%
the error on eigenvectors in the approximation (\ref{Physic_approximation})
where $\ell _{k},$ $n_{k}$ are the optimal arguments in (\ref{num4}). The
optimal error on eigenvectors is then%
\begin{equation}
er_{vector}=\min_{k\in L_{125}^{\ast +}}er_{vector}\left( k\right) .
\label{min_error_fixed_p}
\end{equation}%
\begin{figure}[h]
\par
\begin{center}
\subfigure{
                 \includegraphics[width=8cm]{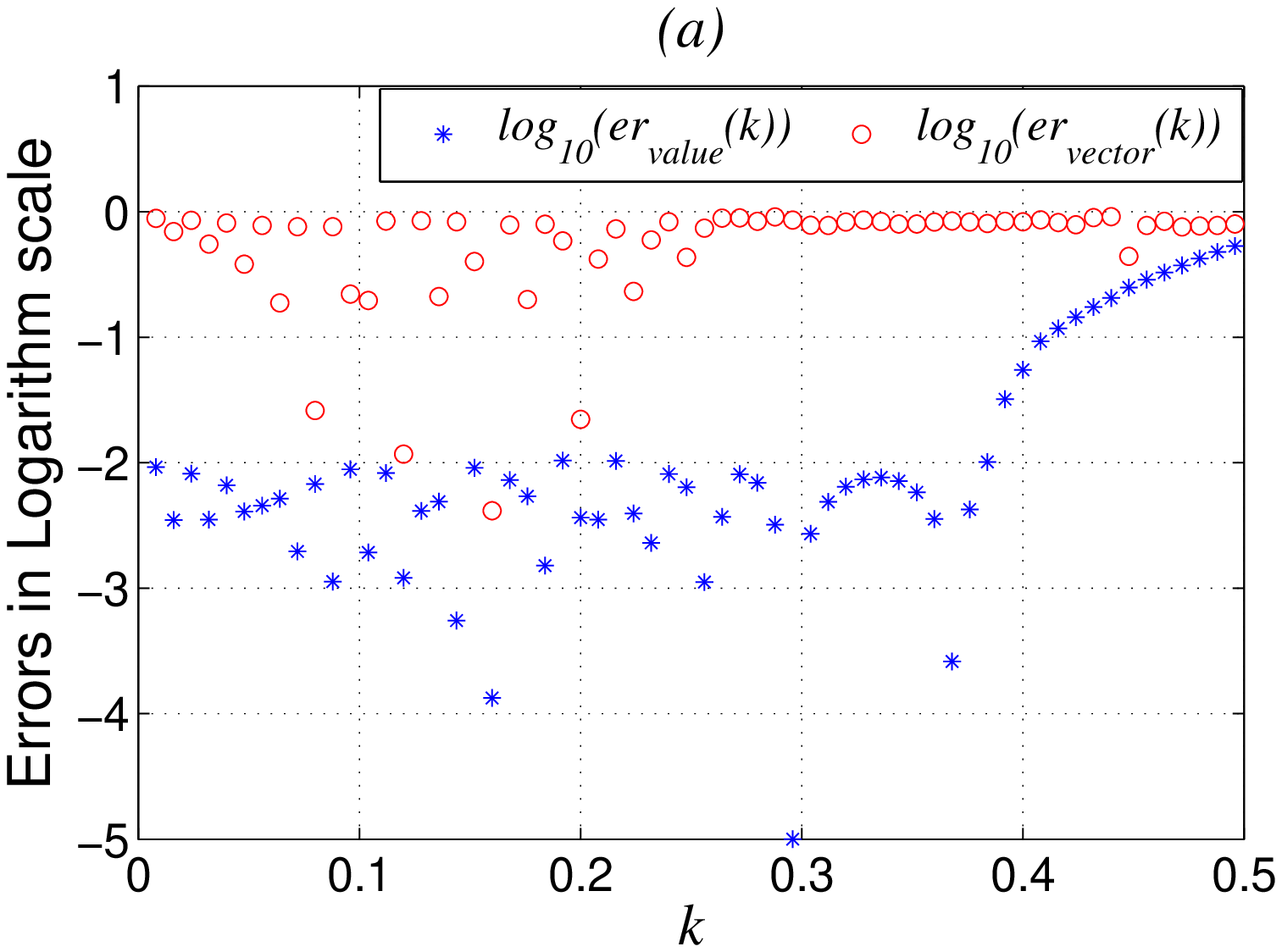}}
\subfigure{
                 \includegraphics[width=8cm]{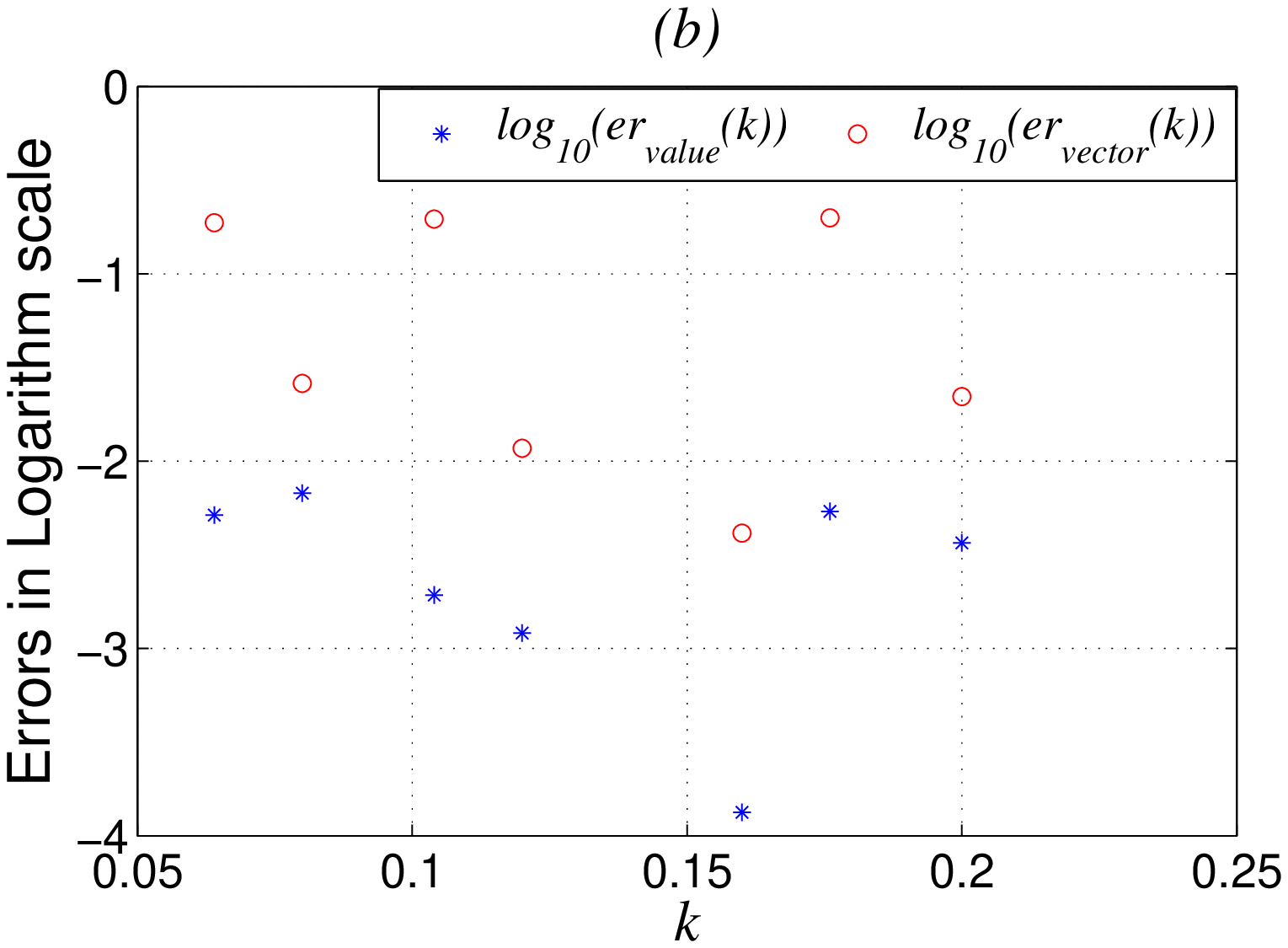}}
\end{center}
\caption{(a) Errors for $p=85$ and $k\in L_{125}^{\ast +}$. (b) Errors for a
selection of $k$ s.t.\textit{\ }$er_{vector}(k)\leq 0.2$. }
\label{m=85-eigenvector1}
\end{figure}

\begin{figure}[h]
\par
\begin{center}
\subfigure{
                 \includegraphics[width=8cm]{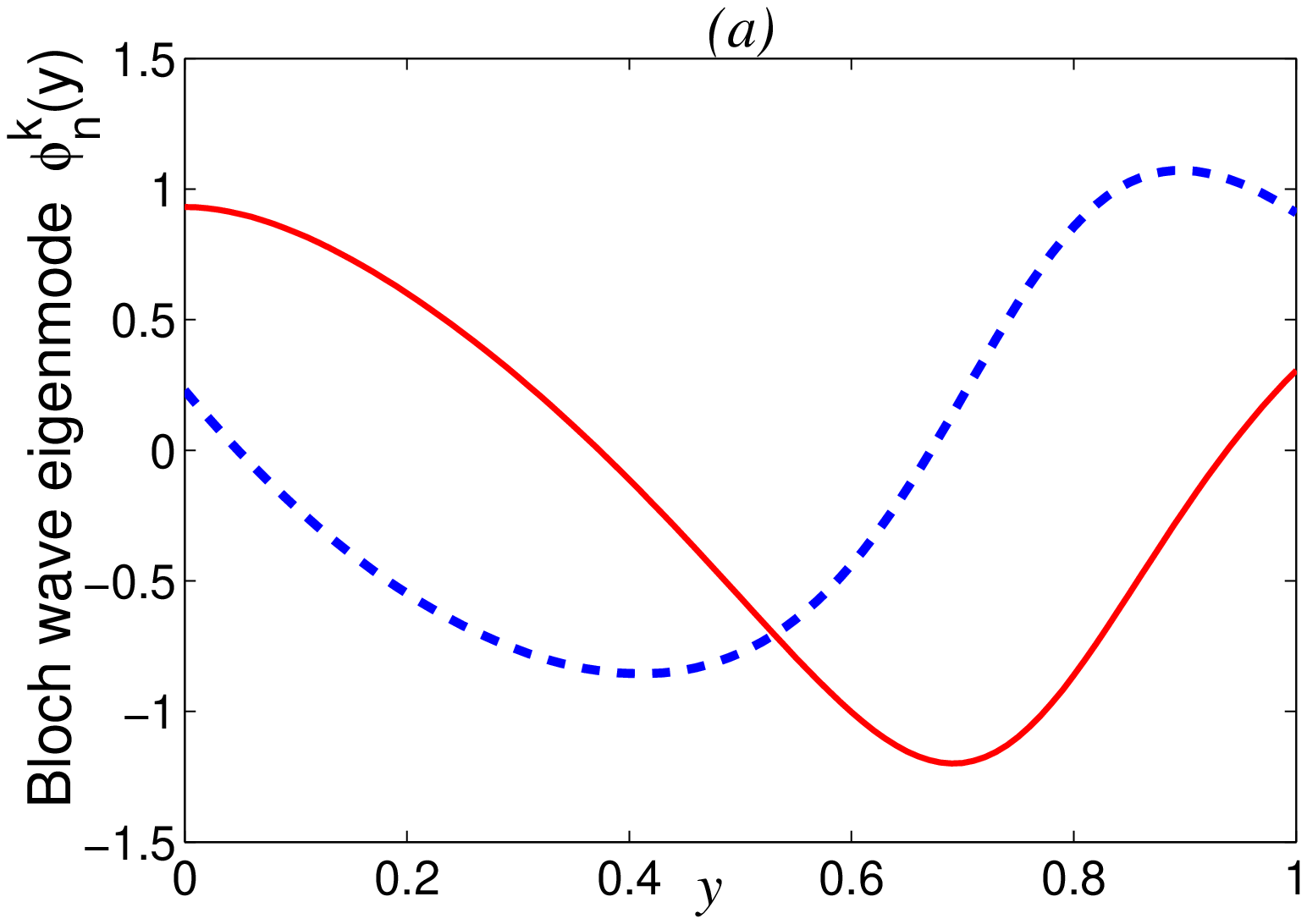}}
\subfigure{
                 \includegraphics[width=8cm]{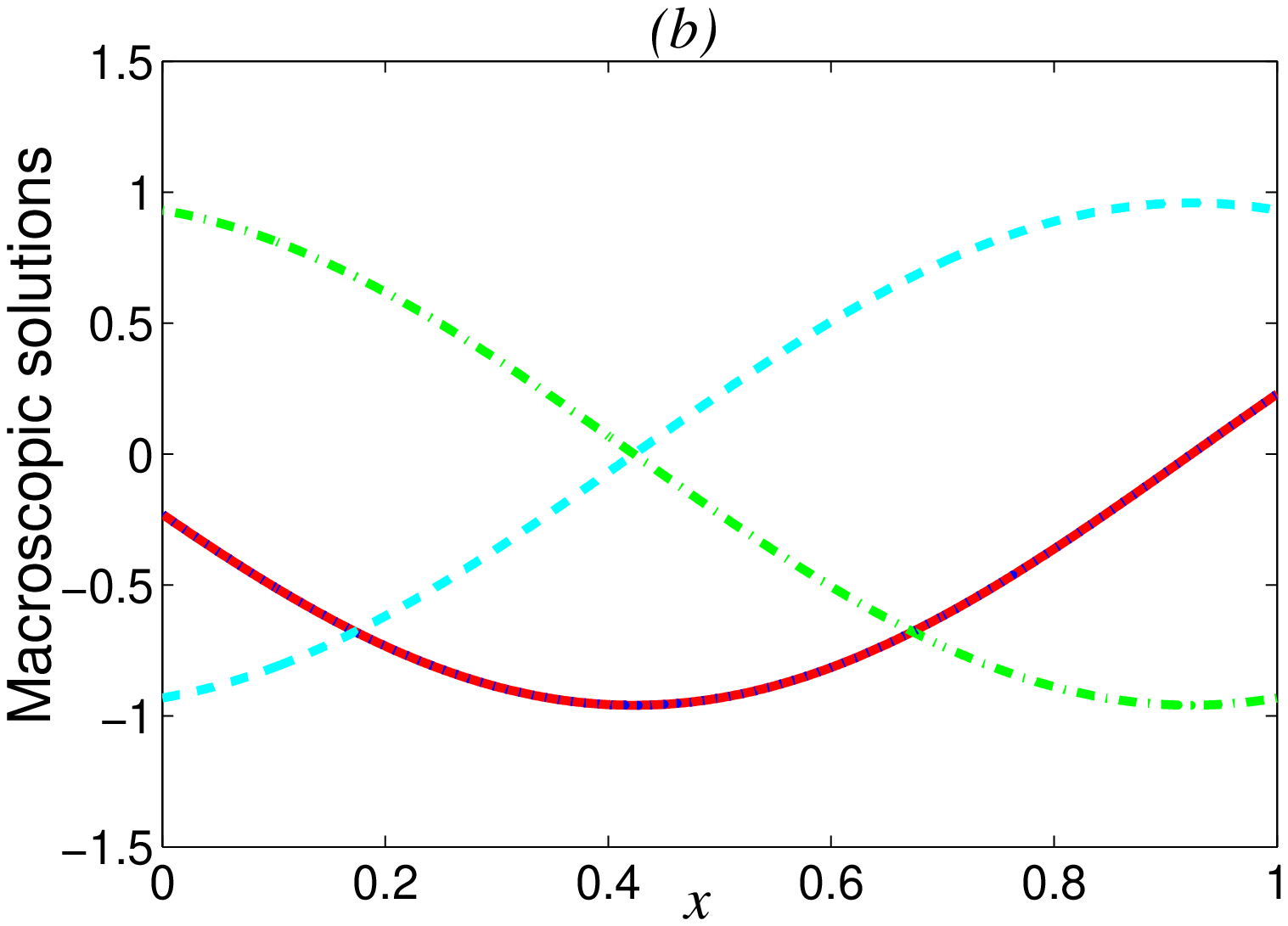}}
\end{center}
\caption{(a) Bloch wave solution $\protect\phi _{n}^{k}$. (b) Macroscopic
solutions $u_{n,\ell }^{k}$ and $u_{n,\ell }^{-k}$.}
\label{m=85}
\end{figure}
Figure \ref{m=85-eigenvector1} (a) shows the distributions of errors $%
er_{value}(k)$ and $er_{vector}(k)$ in logarithmic scale for the index $p=85$
of physical eigenmode with respect to $k$ varying in $L_{125}^{\ast +}$.%
\textbf{\ }The minimal error is reached for $k=0.16,$ $n=2,$ $\ell =17$, $%
\lambda _{n}^{k}=51.1$ and $\lambda ^{1,\ell }=58.9$ yielding the errors $%
er_{value}=10^{-4}$ and $er_{vector}=4.10^{-3}$. Figure \ref%
{m=85-eigenvector1} (b) focuses on values of $k$ such that $%
er_{vector}(k)\leq 0.2$. In Figure \ref{m=85} (a) the real (dashed line) and
the imaginary (solid line) parts of the Bloch wave $\phi _{n}^{k}$ are shown
when Figure \ref{m=85} (b) presents the real (solid line) and the imaginary
(dashed-dotted line) parts of $u_{n,\ell }^{k}$ and also the real (dotted
line) and the imaginary (dashed line) parts of $u_{n,\ell }^{-k}$. In
addition, the physical eigenmode $w_{p}^{\varepsilon }$ and the relative
error vector between $w_{p}^{\varepsilon }$ and $\psi _{n,\ell
}^{\varepsilon ,k}$ are plotted in Figure \ref{m=85-eigenvector} (a) and
(b).
\begin{figure}[h]
\par
\begin{center}
\subfigure{
                 \includegraphics[width=8cm]{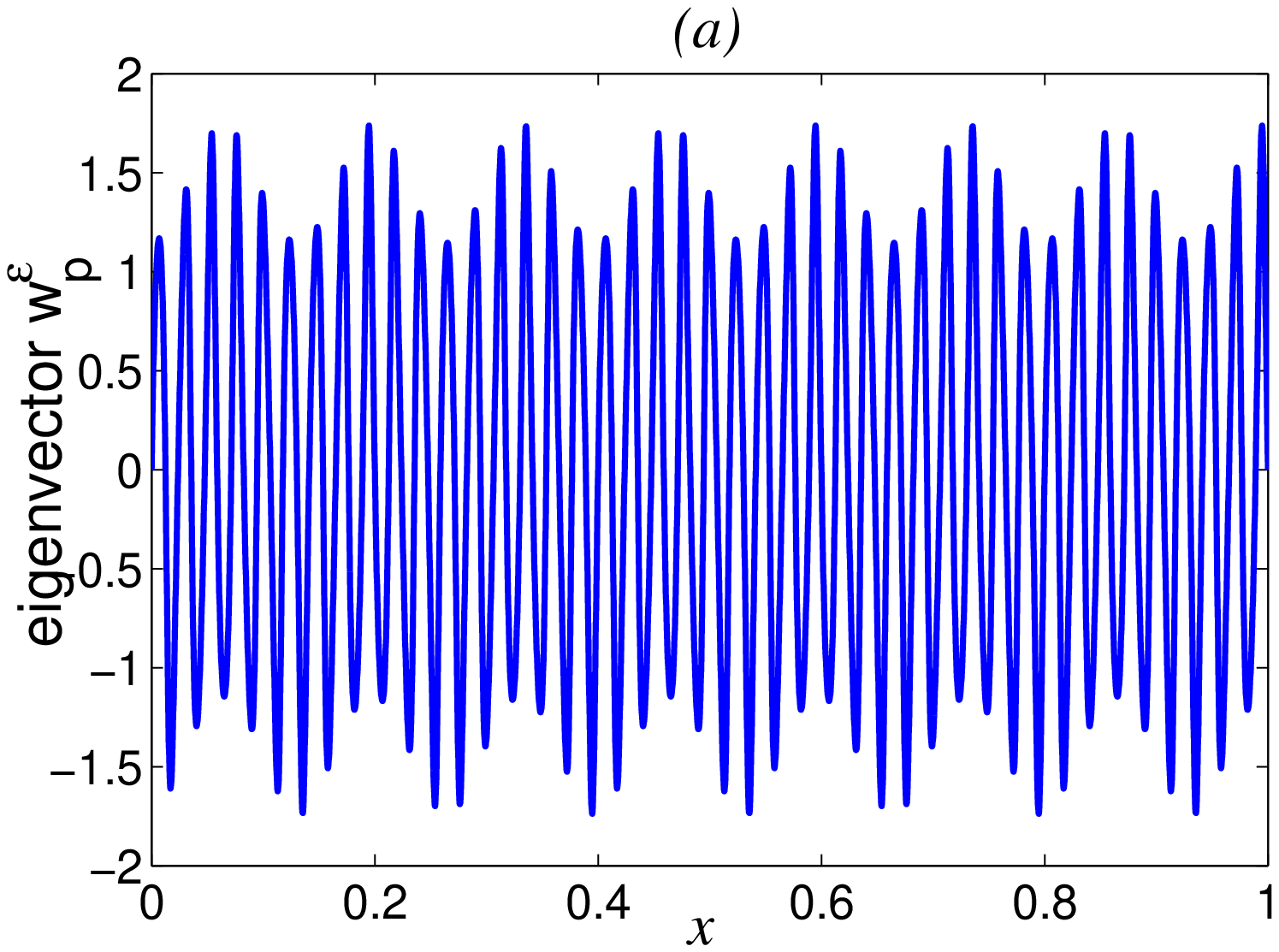}}
\subfigure{
                 \includegraphics[width=8cm]{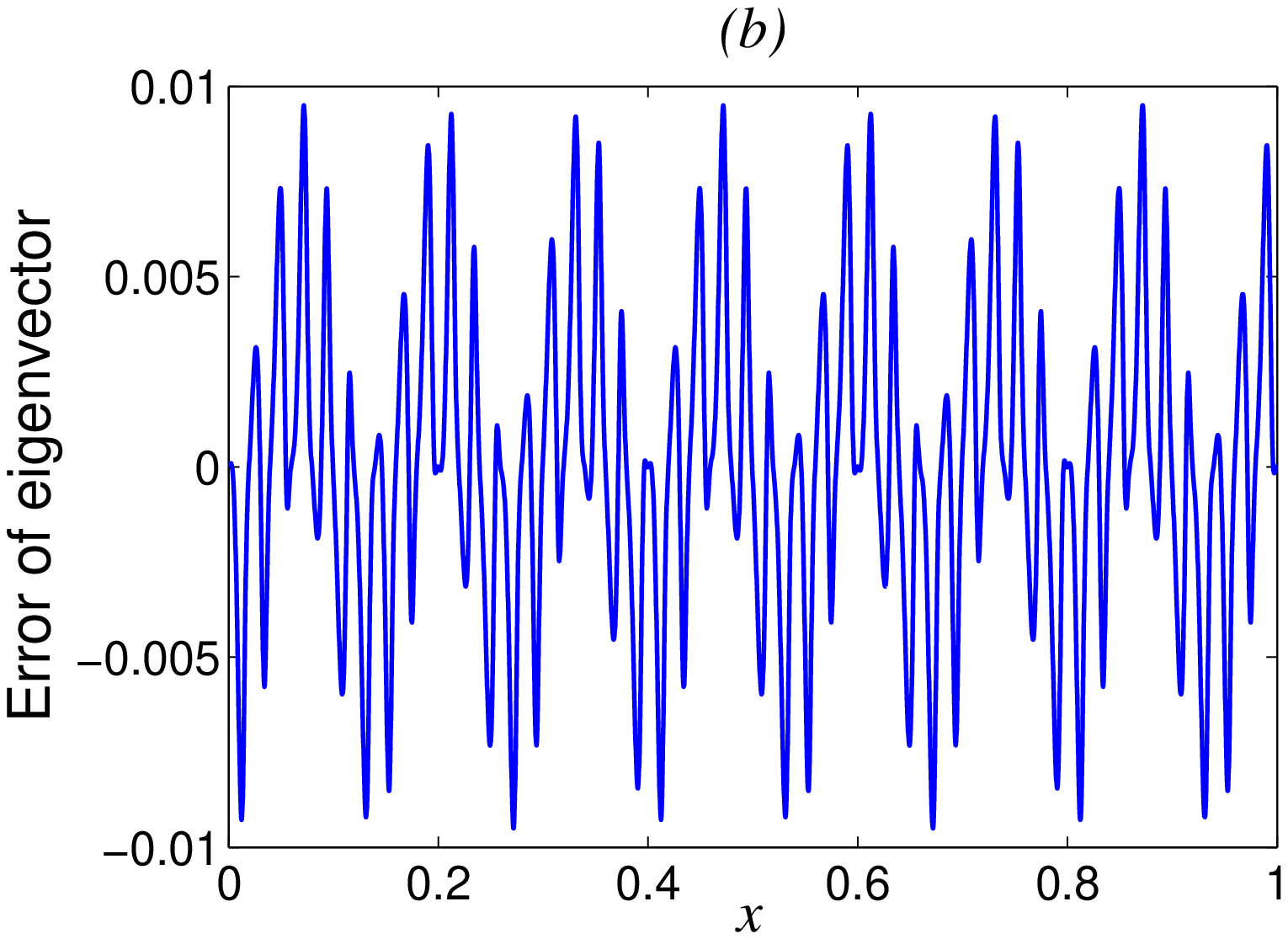}}
\end{center}
\caption{(a) Physical eigenmode $w_{p}^{\protect\varepsilon }$. (b) Relative
error between between $w_{p}^{\protect\varepsilon }$ and $\protect\psi %
_{n,\ell }^{\protect\varepsilon ,k}$. }
\label{m=85-eigenvector}
\end{figure}

\bigskip

\noindent After presenting a detailed study of the approximation of a given
physical mode, i.e. for a single physical mode index $p$, we report
approximation results for the list $\mathcal{J}_0^{\varepsilon }=\left\{
40,...,150\right\} \backslash \left\{ 50\right\} $ of consecutive physical
mode indices. The list starts at $p=40$ corresponding to an intermediary
mode between the low frequency modes approximated by the classical
homogenized method and the high frequency modes considered in this paper.
The index $p=50$ is excluded from the list since the corresponding
eigenvector is evanescent, and as such corresponds to an element of the
boundary spectrum. The previous optimization has been applied to each $p$
yielding errors plotted in logarithm scale in Figure \ref%
{errors-for-a-sequence-of-p} (a).\ The error bounds are $er_{value}\leq
6.10^{-3}$ and $er_{vector}\leq 8.10^{-2}$.

\begin{figure}[h]
\par
\begin{center}
\subfigure{
                \includegraphics[width=8cm]{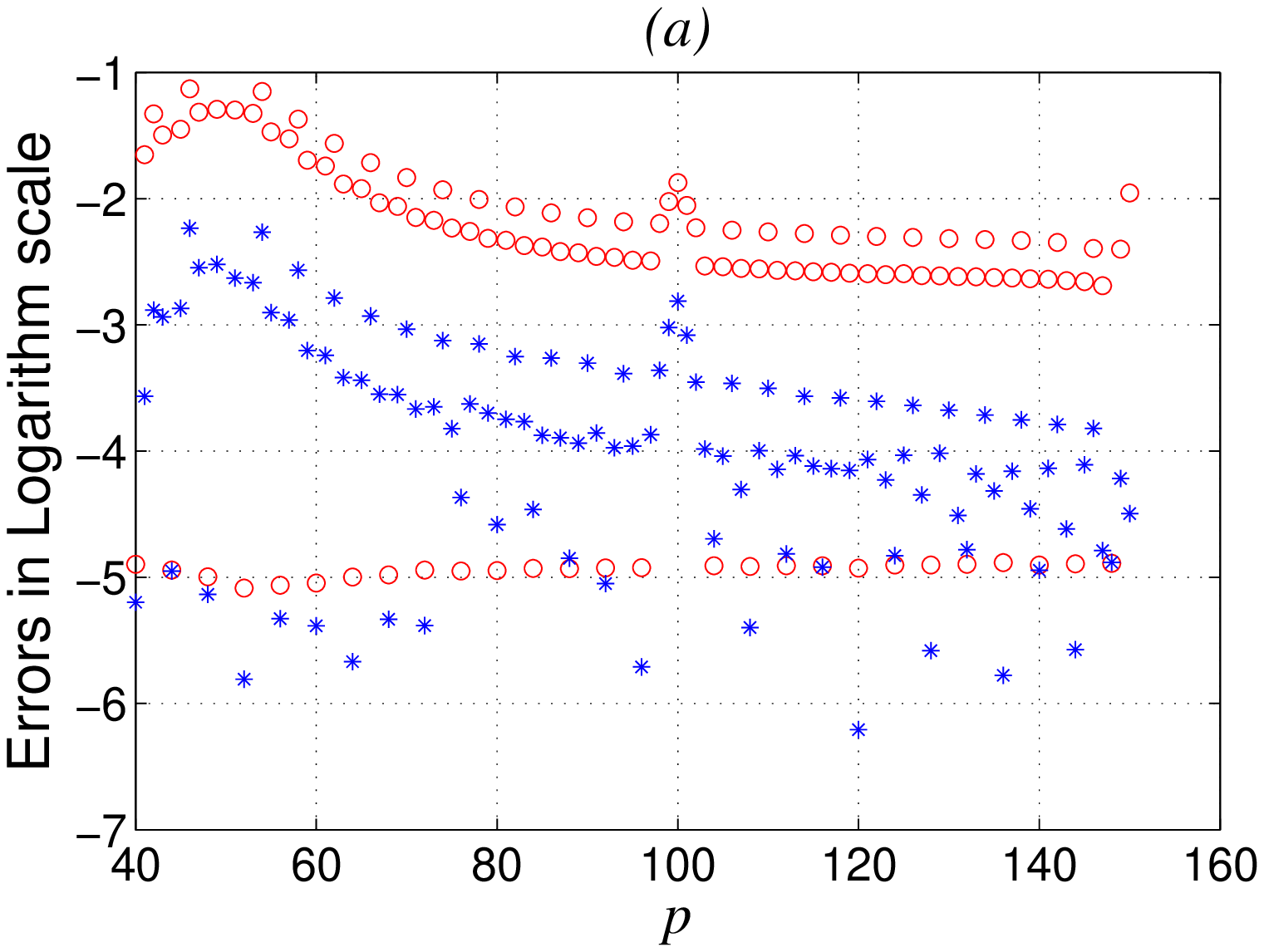}}
\subfigure{
                \includegraphics[width=8cm]{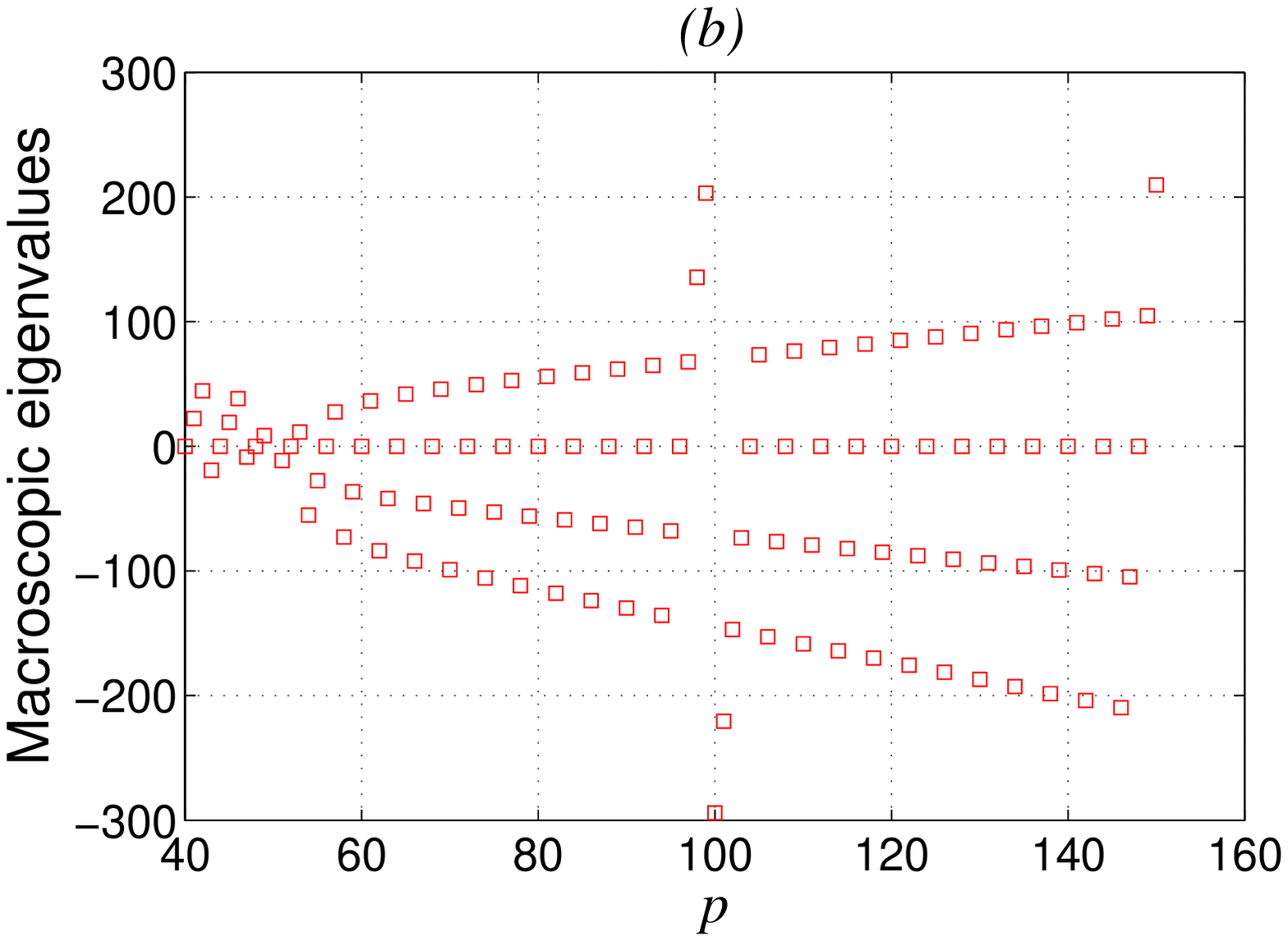}}
\end{center}
\caption{(a) Errors for $p$ varying in $\mathcal{J}_0^\protect\varepsilon$.
(b) Macroscopic eigenvalues.}
\label{errors-for-a-sequence-of-p}
\end{figure}
Globally, the errors start by growing before to decrease except around $%
p=100 $ where they exhibit a peak that we do not explain.\ Figure \ref%
{errors-for-a-sequence-of-p} (b) reports the corresponding macroscopic
eigenvalues $\lambda ^{1,\ell }$. Some of them are close to pairs $(k,n)$
such that $\lambda ^{1,\ell }$ vanishes as discussed in Remark \ref%
{exact_solution}; their relative errors on eigenvalues are in the order of $%
10^{-5}$. A way to answer the question in Remark \ref{exact_solution} is to
decrease the step $\Delta _{k}$ and see if all error decrease. A detailed
presentation is made in the table below for two indices, namely $p=66$
related to an eigenvalue in the beginning of the high frequency spectrum and
$p=102$ corresponding to one of the large errors. In both cases, the error
diminishes as the step $\Delta _{k}$ is reduced from 8e-3 to 3e-3.

\begin{center}
\begin{tabular}{|l|l|l|l|l|l|l|}
\hline
$\Delta _{k}$ & $p$ & $k$ & $n$ & $\lambda ^{1,\ell }$ & $er_{value}$ & $%
er_{vector}$ \\ \hline
8.0e-3 & 66 & 2.16e-1 & 2 & -92 & 1.2e-3 & 1.9e-2 \\ \hline
3.0e-3 & 66 & 3.4e-1 & 2 & 21.7 & 9.0e-5 & 5.3e-3 \\ \hline
8.0e-3 & 102 & 4.0e-2 & 3 & -147 & 4.0e-4 & 5.8e-3 \\ \hline
3.0e-3 & 102 & 1.5e-2 & 3 & 35.9 & 3.0e-5 & 1.4e-3 \\ \hline
\end{tabular}

Table 1: Errors for $\Delta _{k}=8.e-3$ and $3e-3$.
\end{center}

\noindent Figure \ref{improve} (a) is a global view of the errors in
logarithm scale when $\Delta _{k}=8.e-3$ for $90\leq p\leq 110$. It shows
that for this $k$-step a large part of the errors on eigenvalues is in the
range of 1.0e-5 i.e. almost the roundoff error. A measure of the error
reduction is provided in Figure \ref{improve} (b) where the two ratios%
\begin{equation*}
E_{value}=\frac{er_{value}^{\Delta _{k}=3.e-3}}{er_{value}^{\Delta
_{k}=8.e-3}}\text{ and }E_{vector}=\frac{er_{value}^{\Delta _{k}=3.e-3}}{%
er_{vector}^{\Delta _{k}=8.e-3}}
\end{equation*}%
of error reduction are represented in logarithmic scale.

\begin{figure}[h]
\par
\begin{center}
\subfigure{
                \includegraphics[width=8cm]{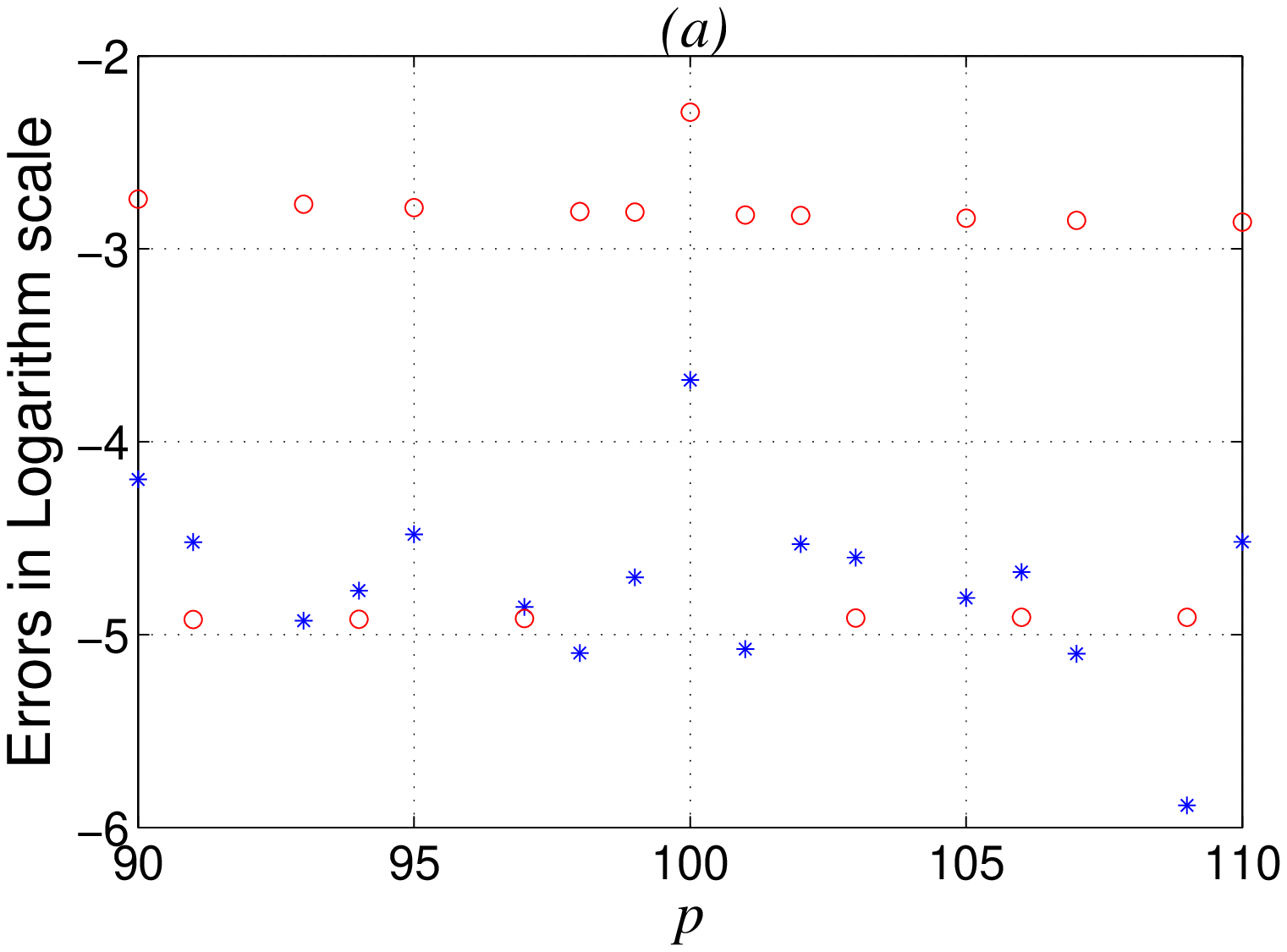}}
\subfigure{
                \includegraphics[width=8cm]{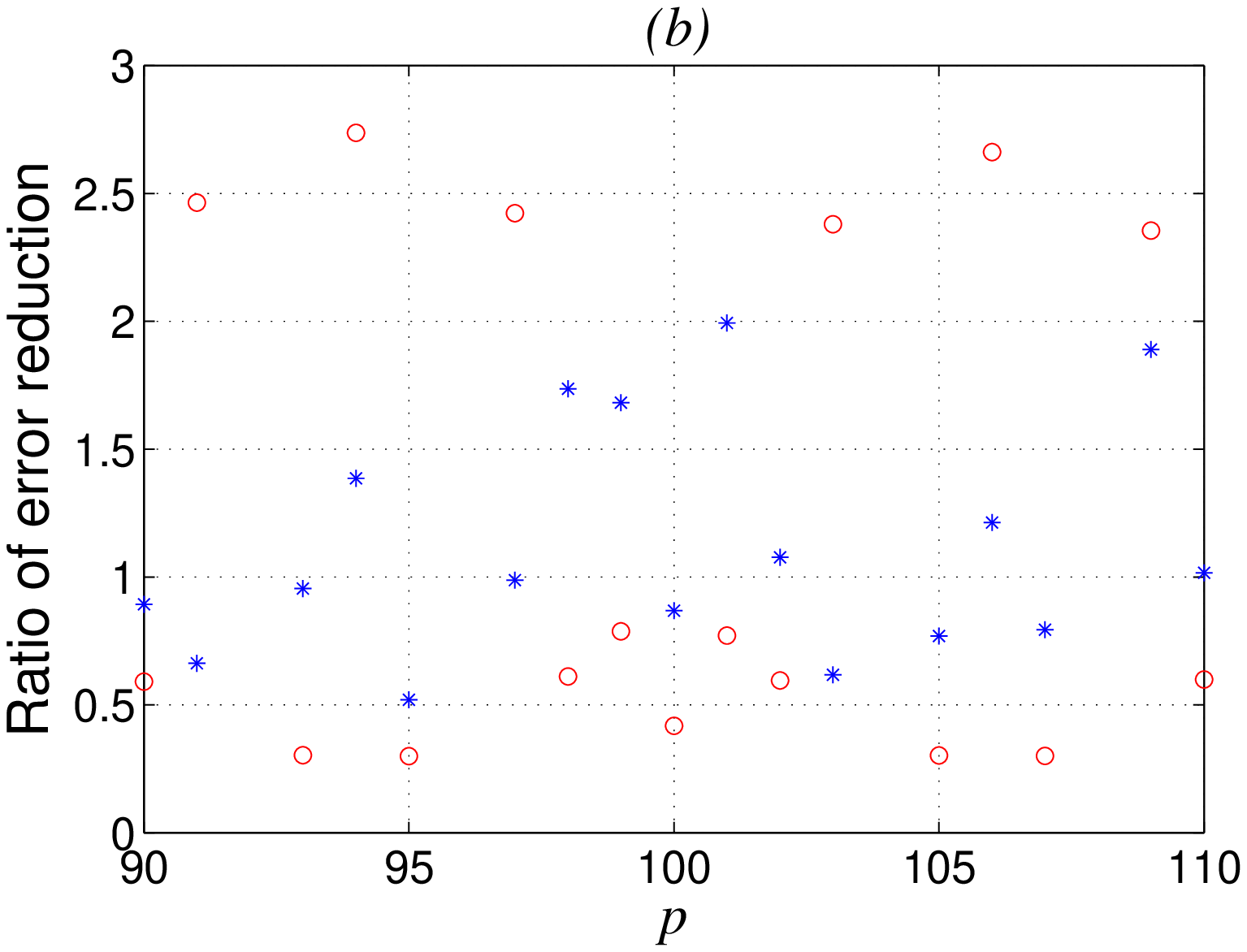}}
\end{center}
\caption{(a) Error of approximation for $\Delta _{k}=3.0e-3$. (b) Ratios $%
E_{value}$ and $E_{vector}$ of error reduction.}
\label{improve}
\end{figure}

\subsection{The modeling problem\label{problem2}}

The modeling problem is reciprocal to the previous one. It consists in
fixing a period $\varepsilon $ as well as the parameters $(k,n)$ of a Bloch
mode and to search if there exists $\ell \in J_{n}^{k}$ such that $(\gamma
_{n,\ell }^{\varepsilon ,k},\psi _{n,\ell }^{\varepsilon ,k})$ is close from
a physical mode or in other words if it is almost a solution to the physical
spectral problem i.e. if%
\begin{equation}
\varepsilon ^{2}P^{\varepsilon }\psi _{n,\ell }^{\varepsilon ,k}-\gamma
_{n,\ell }^{\varepsilon ,k}\psi _{n,\ell }^{\varepsilon ,k}=O(\varepsilon )%
\text{ in \ }\Omega .  \label{num-phys}
\end{equation}%
Posing for $\ell \in J_{n}^{k}$,%
\begin{equation}
F_{n}^{\varepsilon ,k}(\ell )=\frac{\left\Vert \varepsilon
^{2}P^{\varepsilon }\psi _{n,\ell }^{\varepsilon ,k}-\gamma _{n,\ell
}^{\varepsilon ,k}\psi _{n,\ell }^{\varepsilon ,k}\right\Vert _{L^{2}\left(
\Omega \right) }}{\left\Vert \gamma _{n,\ell }^{\varepsilon ,k}\psi _{n,\ell
}^{\varepsilon ,k}\right\Vert _{L^{2}\left( \Omega \right) }}
\label{relative_error}
\end{equation}%
the modeling problem relies to the minimization problem $F_{n}^{\varepsilon
,k}(\ell _{0})=\min\limits_{_{\ell \in J_{n}^{k}}}F_{n}^{\varepsilon
,k}(\ell )$. If the minimum is small enough, $(\gamma _{n,\ell
_{0}}^{\varepsilon ,k},\psi _{n,\ell _{0}}^{\varepsilon ,k})$ is close from
a physical eigenelement and it is a solution to the modeling problem. A
subsequent problem is to identify the corresponding physical eigenelement.
This is done be minimizing the errors $er_{value}$ and $er_{vector}$
introduced in the previous section but considered as depending on the
parameter $p\in \mathcal{J}^{\varepsilon }$ instead of $k$. Two illustrative
examples are reported in the table below, one yielding $\lambda ^{1,\ell }=0$
and the other $\lambda ^{1,\ell }\neq 0$. The solution $\psi _{n,\ell
}^{\varepsilon ,k}$ and the relative error between $\psi _{n,\ell
}^{\varepsilon ,k}$ and $w_{p}^{\varepsilon }$ are reported in Figures \ref%
{problem2-1} (a) and (b).

\begin{center}
\begin{tabular}{|l|l|l|l|l|l|l|l|}
\hline
$k$ & $n$ & $\lambda _{n}^{k}$ & $F_{n}^{\varepsilon ,k}(\ell )$ & $\lambda
^{1,\ell }$ & $p$ & $er_{value}$ & $er_{value}$ \\ \hline
1.6e-1 & 2 & 5.11e1 & 8.9e-3 & 0 & 84 & 3.4e-5 & 2.1e-5 \\ \hline
3.52e-1 & 2 & 3.14e1 & 4.5e-2 & -8.55 & 65 & 1.5e-2 & 4.3e-3 \\ \hline
\end{tabular}

Table 2: Results for the modeling problem
\end{center}

\begin{figure}[h]
\par
\begin{center}
\subfigure{
                 \includegraphics[width=8cm]{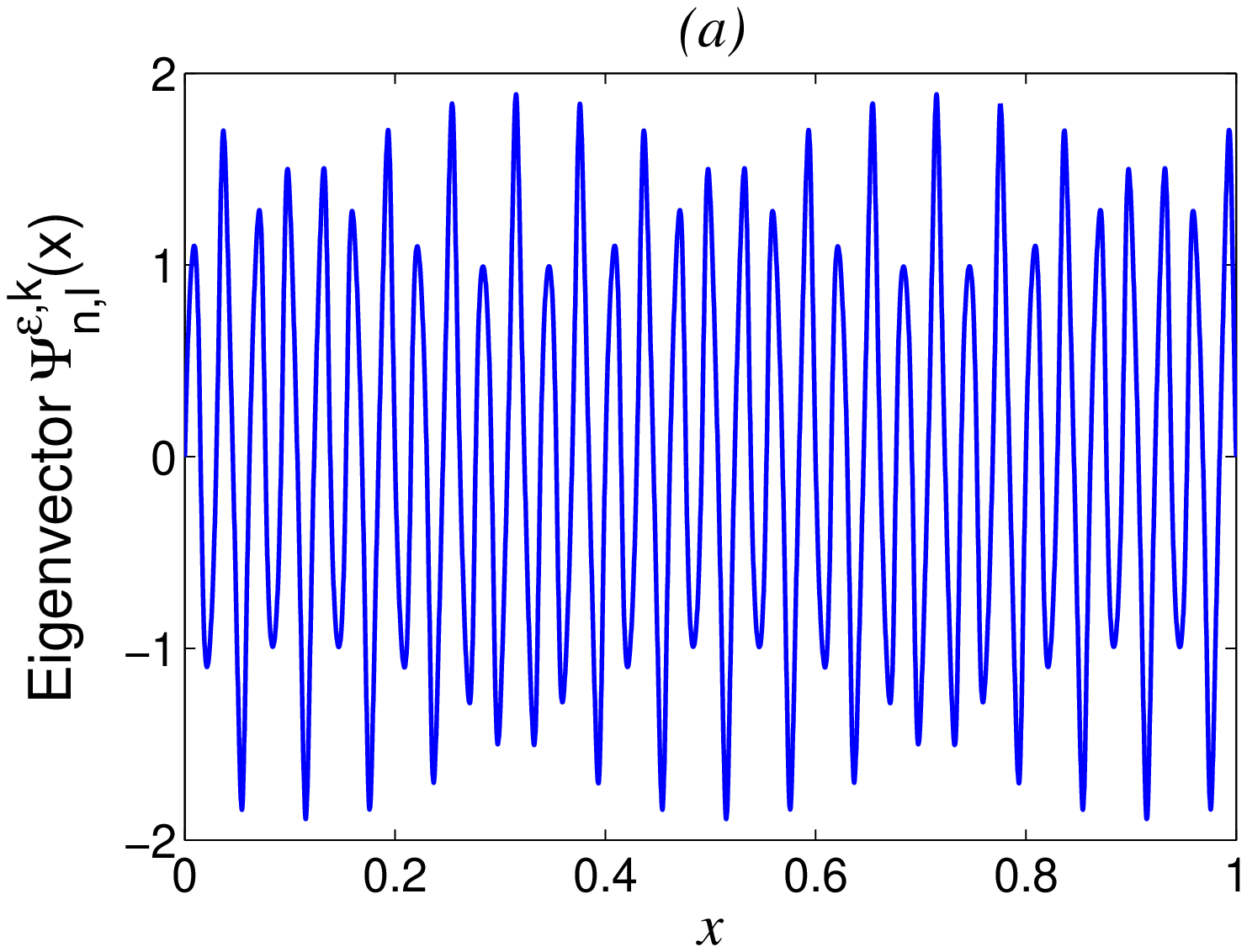}}
\subfigure{
                \includegraphics[width=8cm]{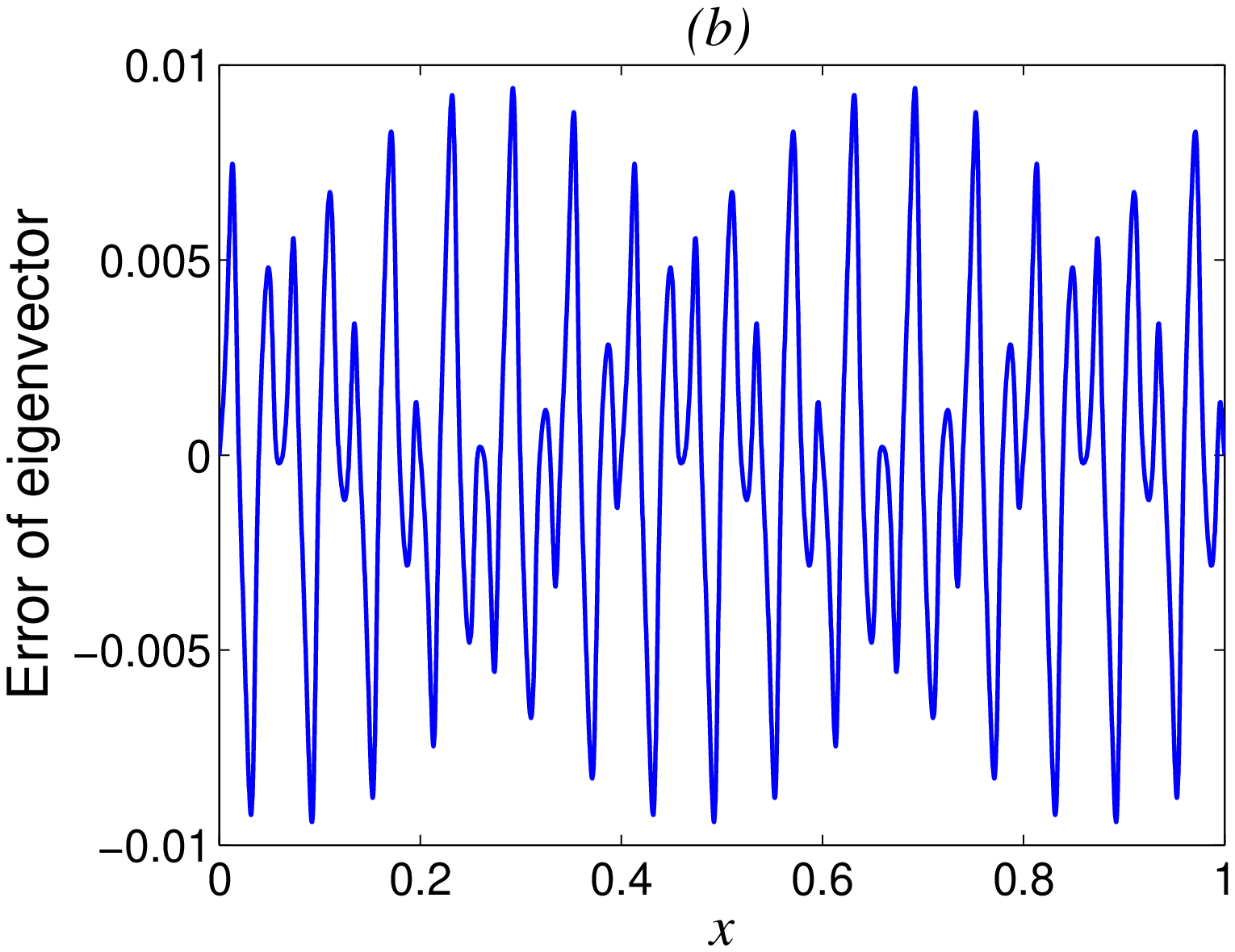}}
\end{center}
\caption{\ (a) Two-scale eigenmode $\protect\psi _{n,\ell }^{\protect%
\varepsilon ,k}$. (b) Relative error vector between $\protect\psi _{n,\ell
}^{\protect\varepsilon ,k}$ and $w_{p}^{\protect\varepsilon }$.}
\label{problem2-1}
\end{figure}
\noindent Additional results for $k=3.52e-1$ with $n=\left\{
1,...,15\right\} $ are reported in Figures \ref{problem2-3} (a) and (b)
showing $\lambda ^{1,\ell }$ and $\gamma _{n,\ell }^{k}$ respectively.

\begin{figure}[h]
\par
\begin{center}
\subfigure{
            \includegraphics[width=8cm]{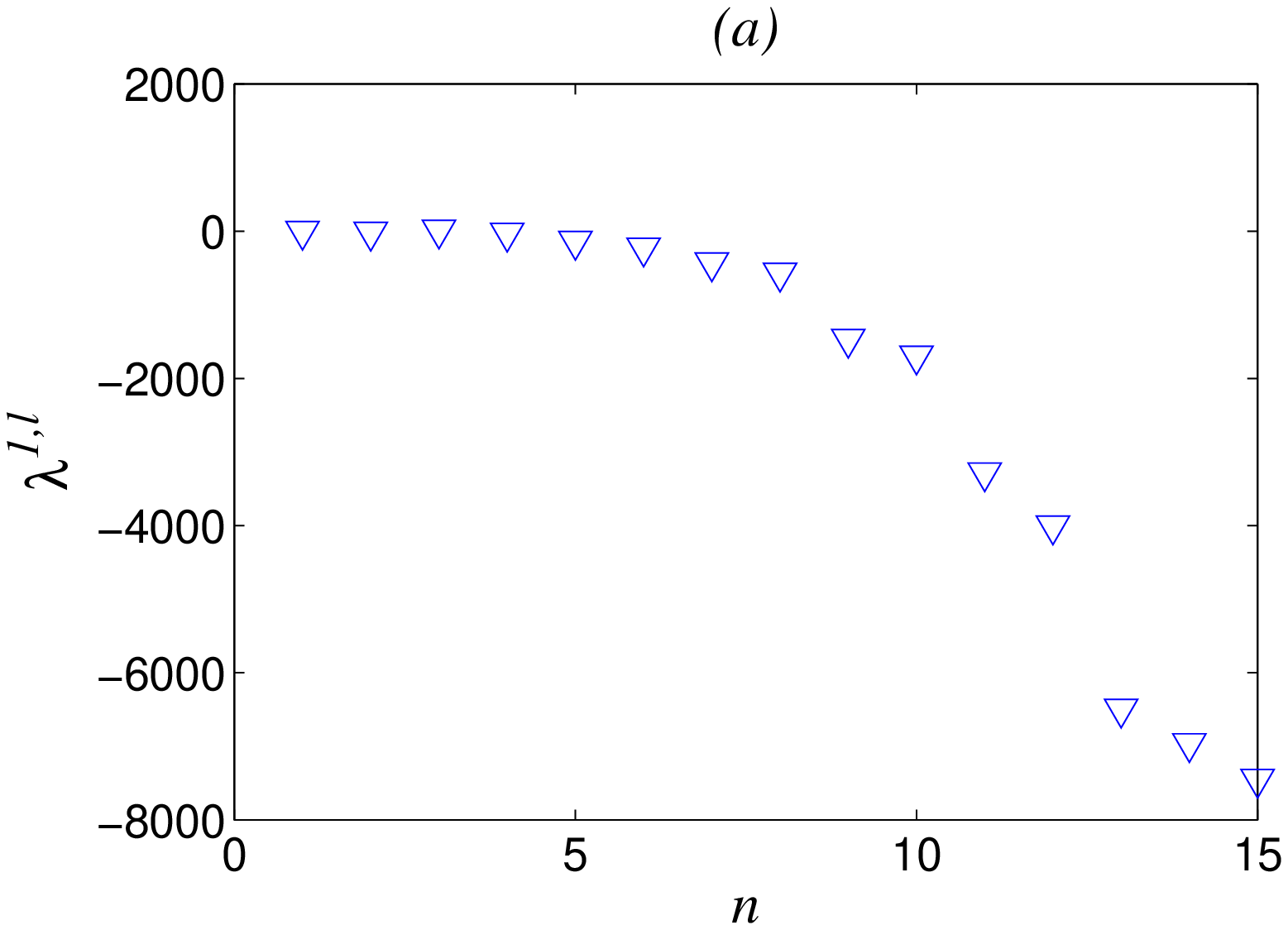}}
\subfigure{
            \includegraphics[width=8cm]{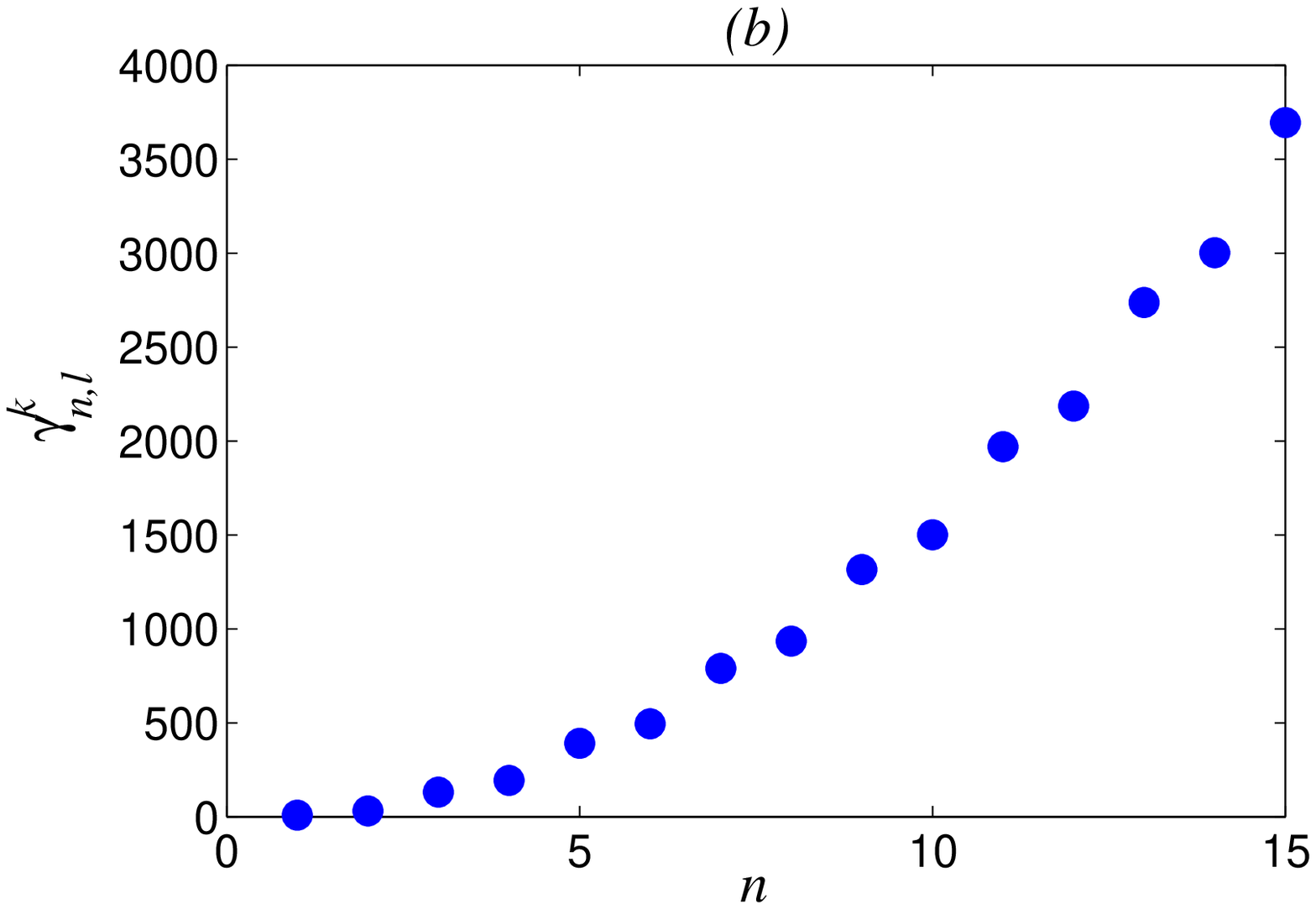}}
\end{center}
\caption{(a) $\protect\lambda ^{1,\ell }$ with respect to $n$. (b) $\protect%
\gamma _{n,\ell }^{k}$ with respect to $n$.}
\label{problem2-3}
\end{figure}

\subsection{Order of convergence\label{problem3}}

For a given pair $k$ and $n\in J^{k}$, we investigate the order of
convergence of the errors $er_{value}$ and $er_{vector}$ when the number of
cells increases. To follow the convergence result, the sequence of periods $%
\varepsilon $ is in fact a subsequence $\varepsilon _{h}$ satisfying%
\begin{equation*}
\frac{1}{\varepsilon _{h}}=\frac{h+l}{k}\in \mathbb{N}^{\ast }
\end{equation*}%
with $l\in \left[ 0,1\right) $ and for a sequence of $h\in
\mathbb{N}
^{\ast }$. Table 3 summarizes the results for $k=0.3$, $l=0.6$ and $h\in
\left\{ 3,9,15,21\right\} $.

\begin{center}
$%
\begin{tabular}{|c|c|c|c|c|}
\hline
$h$ & $\varepsilon _{h}$ & $er_{value}^{_{h,\ell }}$ & $er_{vector}^{h,l}$ &
$p$ \\ \hline
$3$ & $8.3e-2$ & $4.3e-2$ & $6.3e-3$ & $17$ \\ \hline
$9$ & $3.1e-2$ & $1.6e-2$ & $2.4e-3$ & $45$ \\ \hline
$15$ & $1.9\,1e-2$ & $1.0e-2$ & $1.5e-3$ & $73$ \\ \hline
$21$ & $1.4e-2$ & $7.0e-3$ & $1.0e-3$ & $101$ \\ \hline
\end{tabular}%
$

Table 3: Errors for a decreasing subsequence $\varepsilon _{h}$
\end{center}

\noindent To evaluate the decay rate of the errors, we pose $%
er_{value}^{_{h,\ell }}=c_{value}\left( \varepsilon _{h}\right) ^{q_{value}}$
and $er_{vector}^{_{h,\ell }}=c_{vector}\left( \varepsilon _{h}\right)
^{q_{vector}}$, so the decay rates satisfy%
\begin{equation*}
q_{value}=\frac{\log \left( er_{value}^{_{h,\ell }}/er_{value}^{_{h^{\prime
},\ell }}\right) }{\log \left( \varepsilon _{h}/\varepsilon _{h^{\prime
}}\right) }\text{ and }q_{vector}=\frac{\log \left( er_{vector}^{_{h,\ell
}}/er_{vector}^{_{h^{\prime },\ell }}\right) }{\log \left( \varepsilon
_{h}/\varepsilon _{h^{\prime }}\right) }\text{.}
\end{equation*}%
Using successive results for $h$ and $h^{\prime }$, yields%
\begin{equation*}
q_{value}=\left\{ 0.988\text{, }0.995\text{, }0.985\right\} \approx 1\text{
and }q_{vector}=\left\{ 0.985\text{, }0.993\text{, }0.994\right\} \approx 1
\end{equation*}%
with coefficients%
\begin{equation*}
c_{value}=\left\{ 0.504\text{, }0.518\text{, }0.497\right\} \approx 0.5\text{
and }c_{vector}=\left\{ 0.0734\text{, }0.0755\text{, }0.0757\right\} \approx
0.07\text{.}
\end{equation*}

\bibliographystyle{plain}
\bibliography{Bloch_wave}

\end{document}